\documentclass[11pt]{article}

\usepackage[english]{babel}
\usepackage{comment}
\usepackage[letterpaper,top=2cm,bottom=2cm,left=3cm,right=3cm,marginparwidth=1.75cm]{geometry}

\usepackage{cite}
\usepackage{subcaption}
\usepackage{stmaryrd}
\usepackage{algorithm}
\usepackage{algpseudocode}
\usepackage{amsmath}
\usepackage{graphicx}
\usepackage{listings}

\lstdefinestyle{mystyle}{
    language=Matlab,
    basicstyle=\ttfamily\small,
    keywordstyle=\color{black},
    commentstyle=\color{gray},
    stringstyle=\color{green!50!black},
    showstringspaces=false,
    columns=fullflexible,
    keepspaces=true
}

\usepackage{rotating}
\usepackage[dvipsnames]{xcolor}
\usepackage[utf8]{inputenc} 
\usepackage{amsfonts}
\usepackage{amsthm}
\usepackage{mathtools}
\usepackage{listings}
\usepackage{amssymb}
\usepackage{multirow}
\usepackage{booktabs}
\usepackage{tikz}
\usetikzlibrary{patterns}
\usetikzlibrary{patterns,arrows.meta}
\usepackage[colorlinks=true, allcolors=blue]{hyperref}
\newtheorem{theorem}{Theorem}[section]

\newtheorem{lemma}{Lemma}[section]
\newtheorem{proposition}{Proposition}[section]

\newtheorem{definition}{Definition}[section]

\title{Incremental Column Subset Selection via Conditional Determinantal Point Processes}
\author{Laura Grigori, Zhipeng Xue}
\date{}
\makeatletter
\newcommand\tabcaption{\def\@captype{table}\caption}
\newcommand\figcaption{\def\@captype{figure}\caption}
\makeatother

\newcommand{\be}{\begin{eqnarray}}
	\newcommand{\ee}{\end{eqnarray}}
\newcommand{\beno}{\begin{eqnarray*}}
	\newcommand{\eeno}{\end{eqnarray*}}
\newcommand{\barr}[1]{\begin{array}{#1}}
	\newcommand{\earr}{\end{array}}

\newcommand{\R}{\mathbb{R}}
\newcommand{\Rmn}{\mathbb{R}^{m\times n}}

\newcommand{\Rnn}{\mathbb{R}^{n\times n}}

\DeclareMathOperator*{\argmax}{arg\,max}
\DeclareMathOperator*{\argmin}{arg\,min}
\algrenewcommand\algorithmicrequire{\textbf{Input:}}
\algrenewcommand\algorithmicensure{\textbf{Output:}}
\begin{document}

\maketitle
\begin{abstract}
    Column subset selection aims to seek a small set of representative columns that accurately approximates a given matrix. We study this problem from the perspective of determinantal point processes (DPPs) and their relation to several existing algorithms. To support incremental column selection, we introduce conditional DPPs, which can be applied to any column selection strategy. In particular, we combine it with adaptive randomized pivoting(ARP), and develop a multi-stage ARP (MSARP) algorithm. We establish theoretical guarantees on the expected Frobenius-norm approximation errors for these methods. In addition, we propose a fixed-precision variant of ARP that adaptively determines the number of selected columns. Numerical experiments on column subset selection and Nystr\"{o}m approximation show that MSARP achieves accuracy comparable to standard ARP while enabling incremental selection.

\end{abstract}
\section{Introduction}
Low-rank matrix approximation is a fundamental tool for computational science, machine learning and data analysis\cite{halko_finding_2010}\cite{markovsky2012low}\cite{tropp2023randomizedalgorithmslowrankmatrix}. Given a matrix $A\in \Rmn$, where both $m$ and $n$ are very large, we aim to approximate it by
$$
A \approx M N
$$
where $M\in \R^{m\times k}$, $N\in \R^{k\times n}$ and $k\ll \min\{m,n\}$ is called the target rank. Among the many approaches for low-rank approximation, the column subset selection problem(CSSP) seeks some representative columns of $A$ to form the matrix $M$. In this setting, we denote $M$ as $C$ and the optimal $N$ is $C^{\dagger} A$. Unlike spectral methods such as singular value decomposition(SVD), column selection yields a low-rank approximation 
$$
A \approx C C^{\dagger} A
$$
preserving some structure of the original matrix. Consequently, CSSP has become important in problems as feature selection and experimental design\cite{book_experiment_design}\cite{qian2020subset}\cite{sener2018active}.

Various algorithms have been developed for CSSP. Classical deterministic methods include column-pivoted QR(CPQR) and strong rank-revealing QR\cite{srrqr}. Recently, many randomized approaches for CSSP have been proposed, including leverage score sampling, volume sampling\cite{volume_sampling1}, and a recent adaptive randomized pivoting algorithm (ARP)\cite{cortinovis2025arp}. However, many randomized algorithms are designed as one-shot procedures. Given a target rank $k$, those algorithms select a subset of $k$ columns and terminate. If the subset needs to be enlarged, it is often unclear how the previous selection can be included. This limitation is significant in many practical scenarios where the desired target rank is unknown and the previous selection must be retained. For example, in machine learning applications such as feature selection or active learning\cite{Ash2020Deep}, DPP-based methods are often used to identify a diverse subset of samples for labeling. Since labeling is typically expensive, the labeled samples cannot be discarded. A similar situation arises in the placement of  sensors \cite{hibbard2023randomized} \cite{krause2008near}. If the budget increases, one may consider installing a few more sensors in addition to the already deployed sensors.

In this paper, we establish a posteriori selection strategies. To this end, we study conditional determinantal point processes (conditional DPPs) as a remedy for any initial column selection algorithm and give a bound for the conditional expectation of the approximation errors. In particular, we combine conditional DPPs and ARP and propose two new algorithms, called Conditional Adaptive Random Pivoting (CARP) and Multi-Stage Adaptive Random Pivoting (MSARP), which make ARP more flexible. We establish approximation guaranties for CARP and MSARP. If $k_i$ columns are sampled at the $i$th stage and $d$ is the total number of the selected columns, the resulting error bound takes the form
\[
\prod_{i}(1+k_i)\|A-\llbracket A \rrbracket_t\|_F^2,
\]
where $\llbracket A \rrbracket_d$ is the approximation obtained from the $d$-truncated SVD, 
showing that each incremental sampling stage incurs only a multiplicative penalty. This provides a bridge between some global algorithms and some greedy ones.

The rest of this paper is organized as follows. In section~\ref{sec:framework}, we present a DPP-based perspective on column subset selection. We first review DPPs and volume sampling, and then establish their connections with adaptive randomized pivoting (ARP) and randomly pivoted Cholesky (RPCholesky). We further discuss several related column selection methods, including strong RRQR, column pivotd QR (CPQR), leverage score sampling, nuclear maximization and the algorithms of Osinsky and Stewart. In section~\ref{sec:conditional_dpp}, we introduce conditional determinantal point processes, which can be applied to any column selection strategy. In particular, we combine it and adaptive randomized pivoting(ARP), and develop a posterior processing method, conditional ARP (CARP) and a multi-stage ARP (MSARP) algorithm. We establish theoretical guarantees on the expected Frobenius-norm approximation errors for these methods. We prove that multi-stage ARP inherits the approximation guarantees of ARP. The resulting analysis reveals a connection between global DPP-based sampling strategies and some greedy algorithms. In section~\ref{sec:arp_tol}, we propose a fixed-precision variant of ARP that adaptively determines the number of selected columns. In section~\ref{sec:numerics}, we show the results of some numerical experiments on columns subset selection problems and Nystr\"{o}m approximation to validate the theoretical guarantees of multi-stage ARP.

\section{A Unifying Perspective}
\label{sec:framework}
In this section, we provide a determinantal point process (DPP)-based perspective to interpret several existing algorithms. We first recall the definitions of DPP and its variants, including a relevant algorithm, volume sampling. We then show the connections between the recent algorithm adaptive randomized pivoting (ARP) and a projection DPP. In addition, we formalize the joint probability of randomly pivoted Cholesky (RPCholesky) and compare it with that of a fixed-size DPP. Finally, we also discuss several other algorithms, including strong rank-revealing QR (strong RRQR), column-pivoted QR (CPQR), levearge score sampling and the derandomized version of ARP (Algorithm 1 in \cite{OSINSKY2025359}). 
\subsection{Determinantal Point Process (DPP) and Volume Sampling}
Firstly, we briefly recall the determinantal point process (DPP) and its variant, volume sampling. These two techniques are used to obtain a representative and diverse subset of a large data set $X$. We focus here only on the case that is relevant to column subset selection. Thus, we let the ground set $X$ to be $\{1,\dots,n\}$, where $n$ is the number of columns of $A$.

\begin{definition}[General DPP]
A \emph{determinantal point process (DPP)} with kernel $K\in\mathbb{R}^{n\times n}$ (symmetric positive semidefinite, $0\preceq K\preceq I$) is a subset of random indices $S\subseteq \{1,\dots,n\}$ such that for any $T\subseteq \{1,\dots,n\}$
\[
\mathbb{P}(T\subseteq S)=\det\bigl(K(T,T)\bigr),
\]
We denote it by $S \sim DPP(K).$
\end{definition}

However, the condition $0\preceq K\preceq I$ in the above definition can be rather restrictive in practice. Moreover, the cardinality of the sampled set $S$ is a random variable, which can cause some inconvenience in some applications. To address these limitations, alternative formulations of DPPs have been proposed. In particular, the \emph{$L$-ensemble DPP} provides a more flexible variant through a positive semidefinite matrix $L$, removing the spectral constraint $0\preceq K\preceq I$. In addition, \emph{fixed-size DPP} (also known as $k$-DPP) allows one to sample subsets with a prescribed cardinality. We briefly introduce these two variants below.

\begin{definition}[L-ensemble DPP and fixed-size DPP]
Let $L\in \Rnn$ be a symmetric positive semidefinite matrix. An $L$-ensemble DPP is a subset of random indices $S\subseteq \{1,\dots,n\}$ defined by
\[
\mathbb{P}(S)\propto \det\bigl(L(S,S)\bigr).
\] 
We denote it as $S \sim DPP_L(L).$

The fixed-size DPP (or $k$-DPP) is the distribution of $S\sim DPP_L(L)$ conditioned on $|S|=k$, i.e.,
\[
\mathbb{P}(S)\propto \det\bigl(L(S,S)\bigr), \qquad |S|=k.
\]
We denote it by $S \sim k\text{-}\mathrm{DPP}_L(L).$
\end{definition}

The L-ensemble DPP is a general DPP and the corresponding kernel $K$ is related to $L$ by
\[
K = L (L + I)^{-1}.
\] For $S\sim DPP(K)$, although the cardinality of $S$ is random, the distribution of $|S|$ can be characterized by the eigenvalues of $K$. We introduce the following proposition (which corresponds to Theorem 7 in \cite{BenHough2006}. 
\begin{proposition}\label{prop:size_of_s}
Given a kernel $K\in\mathbb{R}^{n\times n}$ (symmetric positive semidefinite, $0\preceq K\preceq I$) with rank $r$, if $S \sim DPP(K)$, then we have 
\[
|S| = \sum\limits_{i=1}^{r} B_i(\lambda_i),
\]
where $B_i\sim\mathrm{Bernoulli}(\lambda_i)$ are independent random variables. 
\end{proposition}

We now recall volume sampling, which was proposed in \cite{volume_sampling1}. For CSSP of $A$, it defines a distribution over all subsets of columns $A(:,S)$ with $|S| = k.$

\begin{definition}[Volume sampling]
Let $A \in \mathbb{R}^{m\times n}$ and let $k \le \operatorname{rank}(A)$. 
The $k$-volume sampling distribution $\mathrm{VS}_k(A)$ over subsets 
$S \subset \{1,\dots,n\}$ with $|S|=k$ is defined by
\[
\mathbb{P}(S) \propto \det(A^T A(S,S)).
\]
\end{definition}

From a geometric perspective, $\det(A^TA(S,S))$ equals the squared $k$-dimensional volume 
of the parallelepiped spanned by the selected columns. It is easy to see that 
if we set $L = A^T A$, then $k$-$\mathrm{DPP}_L(L)$ and $\mathrm{VS}_k(A)$ define 
the same distribution, $i.e. $,
\be
k\text{-}\mathrm{DPP}_L(A^TA) = \mathrm{VS}_k(A).
\ee
This provides one connection between DPPs and volume sampling, as mentioned in \cite{epperly2026adaptiverandomizedpivotingvolume}. 

We introduce a special case to end this section. When the kernel of a DPP is set as an orthogonal projector, it is referred to as a projection DPP. 

\begin{definition}[Projection DPP]
A projection DPP is a DPP with associated kernel $K = VV^T$ where $V\in \R^{n\times k}$ is an orthogonal matrix, $i.e.$ $V^TV = I$. 
\end{definition}

By Proposition~\ref{prop:size_of_s}, we know that if $S$ is sampled from a projection DPP, $i.e.\ S\sim DPP(VV^T)$, then the cardinality of $S$ must be $k$. This is useful for later derivation.
\subsection{Connection between ARP, projection DPP, and volume sampling}
We introduce the ARP algorithm. For the CSSP of the matrix $A\in \R^{m\times n}$, it takes as input an orthonormal matrix $V\in \R^{n\times k}$ that satisfies
\be
\label{eq:rowspaceapprox}
\|A - AVV^T\|_F \approx 0,
\ee
and identifies a set of indices $S$ of informative columns of the matrix $A$ with $|S| = k$. For such a $V$, the best choice is $V_{opt}$, the first $k$ right singular vectors obtained by computing the SVD of $A$. However, it is expensive to compute an SVD for a large-scale matrix $A$. In \cite{cortinovis2025arp}, Cortinovis and Kressner suggest using a randomized rangefinder. They propose to compute the sketch of the matrix $A^T$, $i.e.\ A^T \Omega$, where $\Omega \in \R^{m \times k}$ is a random sketching matrix, and then obtain $V = orth(A^T \Omega)$. Then randomly pivoted QR method on $V^T$ is used to select a subset of $k$ columns $A$. At the $i$ iteration the $j$th column with probability proportional to $\|V(j,:)\|_2^2$ and $V$ is updated by orthogonalizing all columns of $V^T$ against the chosen vector:
\begin{subequations}
\begin{align}
    &\mathbb{P}(s_i = j) \propto \|V(j,:)\|_2^2,\\
    &V^T \leftarrow (I - V(s_i,:)^{\dagger}V(s_i,:))V^T.
\end{align}
\end{subequations}
The ARP algorithm is presented in Algorithm \ref{alg:ARPnew}. We add the final updated $V^{(k)}$ in the output for later derivation. Concerning the quality of the selected columns, if they are indexed by $S$, it is proved in \cite{cortinovis2025arp} that the oblique projector $E_{S}(V^TE_{S})^{-1}V^T$ works well for approximating the matrix $A$, as described in the following theorem.

\begin{theorem}[Theorem 2.4 in \cite{cortinovis2025arp}] \label{thm:arp}
    Let $A \in \R^{m \times n}$ and let $V \in \R^{n \times k}$ be an orthonormal basis. Then the random index set $S$ returned by Algorithm~\ref{alg:ARPnew} satisfies
    \begin{equation}\label{eq:arp_bound}
        \mathbb{E}\left [\|A - A(:,S) V(S,:)^{-T} V^T\|_F^2\right ] \leq (k+1) \|A - AVV^T\|_F^2.
    \end{equation}
\end{theorem}

\begin{algorithm}[H]
\caption{Adaptive Randomized Pivoting for CSSP (ARP)}
\label{alg:ARPnew}
\begin{algorithmic}[1]
\Require Matrix $V \in \mathbb{R}^{n \times k}$ with orthonormal columns
\Ensure Indices $S=\{s_1, \ldots, s_k\}$, Matrix $V^{(k)}\in \mathbb{R}^{n\times k}$
\Function{$[V^{(k)},S] =$ ARP-rank}{$V$}
    \State Initialize $S = \emptyset$ and $V^{(0)} = V$
    \For{$i = 1,\ldots,k$}
        \State $p_j = \|V^{(i-1)}(j,i:k)\|_2^2/(k-i+1)$ for $j=1,\ldots,n$
        \State Sample index $s_i$ according to probabilities $p_j$
        \State $S \gets (S ,s_i)$
        \State $V^{(i)} \gets V^{(i-1)} H_i$
        \Comment{$H_i$ zeros out $V^{(i-1)}(s_i,i+1:k)$}
    \EndFor
    \State \Return $V^{(k)},S$
\EndFunction
\end{algorithmic}
\end{algorithm}
    
As observed in Remark~2.2 of~\cite{cortinovis2025arp}, Algorithm \ref{alg:ARPnew} is equivalent to a projection DPP. By the equivalence between volume sampling and fixed-size DPP mentioned above, ARP is also equivalent to $\mathrm{VS}_k(V^T)$, as shown in \cite{epperly2026adaptiverandomizedpivotingvolume}. We provide some more insight here. We first introduce a lemma. 

\begin{lemma}
Let $A \in \mathbb{R}^{m\times n}$ with $\operatorname{rank}(A)=k$ and $V \in \mathbb{R}^{n\times k}$ be the right singular vectors of $A$.
Then
\[
\mathrm{VS}_k(A) = k\text{-}\mathrm{DPP}_L(A^TA) = \mathrm{DPP}(VV^T).
\]

\end{lemma}

\begin{proof}
By Proposition~\ref{prop:size_of_s}, if $S\sim \mathrm{DPP}(VV^T)$, then $|S|=k$. 
Thus both $\mathrm{DPP}(VV^T)$ and $\mathrm{VS}_k(A)$ generate subsets 
$S\subset \{1,\dots,n\}$ with $|S|=k$.

Since $V$ contains the right singular vectors of $A$, we have the eigendecomposition 
\[
A^TA = V\Lambda V^T,
\]
where $\Lambda\in\mathbb{R}^{k\times k}$ is diagonal with positive entries.
Thus, for volume sampling, a column subset indexed by $S$ is selected with probability:
\be
\mathbb{P}(S) \propto \det( V(S,:)\Lambda V(S,:)^T) = \det( (VV^T)(S,S)\Lambda ).
\ee
Since $(VV^T)(S,S)$ and $\Lambda$ are $k\times k$ square matrices, we have
\be
\det( (VV^T)(S,S)\Lambda ) = \det( (VV^T)(S,S) )\det(\Lambda).
\ee
This implies $\mathbb{P}(S)\propto \det( (VV^T)(S,S) )$, which coincides with the probability distribution defined by $DPP(VV^T)$. 
\end{proof}

Therefore, when $A$ has rank $k$, $k$-volume sampling is equivalent to a projection DPP. In general, sampling from $\mathrm{VS}_k(A)$ can be computationally expensive, 
whereas sampling from a projection $\mathrm{DPP}(QQ^T)$ is considerably easier. The ARP algorithm implicitly performs volume sampling on a compressed matrix $B$ satisfying $\mathrm{rank}(B) = k$ and takes as input $V$, the right singular vectors of $B$. It exploits the equivalence between volume sampling and projection DPP to realize volume sampling on $B$ via the projection $\mathrm{DPP}(VV^T)$. If we use $V_{opt}$ as the input to ARP, the corresponding $B$ is $\llbracket A \rrbracket_k$, the approximation obtained from k-truncated SVD of $A$. On the other hand, if we  obtain such a $V$ from the sketch of $A^T$:
\[
V = \mathrm{orth}(A^T\Omega),
\]
the corresponding $B$ is $\Omega^T A$. 

\subsection{Connection between DPP and RPCholesky}
\label{sec:rpc_dpp_connection}

Before introducing the Randomly Pivoted Cholesky (RPCholesky) algorithm, we recall that selecting representative column indices for the CSSP of $A\in \R^{m\times n}$ is equivalent to selecting pivots for the Cholesky factorization of the Gram matrix $K = A^TA\in \Rnn$ (see, e.g., \cite{Higham1990Cholesky,epperly2024gram} for details). This establishes a connection between the pivoted QR and the pivoted Cholesky strategies. Given some indices $S$ such that $K(S,S)$ is invertible, we have 
\be
\|A - A(:,S)A(:,S)^{\dagger} A\|_F^2 = \mathrm{tr}(K - K(:,S)K(S,S)^{-1}K(S,:)).
\ee 

RPCholesky is studied in \cite{RPCholesky}. It operates by iteratively sampling columns proportional to the diagonal entries of the residual matrix. We present it in Algorithm \ref{alg:rpcholesky}. Although RPCholesky is mainly designed for approximating symmetric positive semi-definite matrices, it can be reduced to randomly pivoted QR in CSSP by the Gram correspondence introduced above.

In the following, we adopt the notations of \cite{RPCholesky} and let $K = A^TA$ to compare performing RPCholesky on $K$ (with given $k$) and $k\text{-}\mathrm{DPP}_L(K)$.

\begin{algorithm}[H]
\caption{RPCholesky}
\label{alg:rpcholesky}
\begin{algorithmic}[1]

\Require PSD matrix $K \in \mathbb{R}^{n \times n}$; target rank $k$
\Ensure Indices $S = (s_1, \ldots, s_k)$; matrix $F \in \mathbb{R}^{n \times k}$ such that $\hat{K} = FF^T$
\Function{$[F,S]$=RPCholesky}{$K,k$}
\State Initialize $F \gets 0$ and $d \gets \mathrm{diag}(K)$
\State $S \gets \emptyset$

\For{$i = 1,\ldots,k$}
    \State Sample index $s_i$ with probability $d_i / \sum_{j=1}^n d_j$
    \State $S \gets (S, s_i)$
    
    \State $g \gets K(:, s_i)$ \label{line:compute_g1}
    \State $g \gets g - F(:,1:i-1)\, F(s_i,1:i-1)^T$ \label{line:compute_g2}
    
    \State $F(:,i) \gets g / \sqrt{g(s_i)}$ \label{line:compute_F}
    
    \State $d \gets d - F(:,i)^2$ \label{line:update_d}
\EndFor

\State \Return $F, S$
\EndFunction
\end{algorithmic}
\end{algorithm}

At step $i$, suppose that we have sampled the partial sequence $S_{i-1} = (s_1,\ s_2,\cdots,\ s_{i-1})$. The corresponding Schur complement is given by
\be
K^{(i-1)} = K - K(:,S_{i-1})K(S_{i-1},S_{i-1})^{-1}K(S_{i-1},:), 
\ee
and RPCholesky selects the next pivot $s_i$ with the conditional probability 
\be
\mathbb{P}\{s_i = j \mid s_1,\ s_2,\cdots,\ s_{i-1}\} = \frac{K^{(i-1)}(j,j)}{\text{tr}(K^{(i-1)})}.
\ee
By taking the product of these conditional steps, we can formalize the joint probability of RPCholesky selecting a specific ordered sequence of $k$ pivots, denoted by the tuple $\sigma = (s_1, \dots, s_k)$:
\begin{equation}
\mathbb{P}_{\text{RPC}}(\sigma) = \prod_{i=1}^k \frac{K^{(i-1)}(s_i, s_i)}{\text{tr}(K^{(i-1)})}.
\end{equation} 
By the standard properties of the Cholesky decomposition, the product of the recursive diagonal pivots forms the determinant of the principal submatrix, meaning $\prod_{i=1}^k K^{(i-1)}(s_i, s_i) = \det(K(S,S))$, where $S = \{s_1, \dots, s_k\}$ is the unordered set of indices containing the elements of $\sigma$. Thus, the exact joint probability of the sequence becomes:
\begin{equation}
\label{eq:rpc_proba}
\mathbb{P}_{\text{RPC}}(\sigma) = \frac{\det(K(S,S))}{\prod_{i=1}^k \text{tr}(K^{(i-1)})} \eqqcolon \frac{\det(K(S,S))}{D_K(\sigma)}.
\end{equation}
In the general case when $K$ is an arbitrary positive semi-definite matrix, the denominator $D_K(\sigma)$ is path-dependent and varies across different permutations. This joint probability formula reveals that RPCholesky is a ``trace-penalized'' DPP. Compared with $k\text{-}\mathrm{DPP}_L(K)$, RPCholesky not only selects the indices subset $S$ with large volume $\det(K(S,S))$, but also forces the algorithm to favor sequences of pivots that rapidly decrease the trace of the residual matrix early in the process. Consequently, RPCholesky can be regarded as a variant of fixed-size DPP that uses a  greedy aproach to decrease the trace of the current residual matrix at each iteration.
We then derive the probability of RPCholesky selecting the index subset $S$: 
\begin{equation}
\label{eq:rpc_subset}
\mathbb{P}_{\text{RPC}}(S) = \sum_{\sigma \in \text{Perm}(S)} \mathbb{P}_{\text{RPC}}(\sigma) = \sum_{\sigma \in \text{Perm}(S)} \frac{\det(K(S,S))}{D_K(\sigma)}.
\end{equation}
In a special case where $K = VV^T$ and $V \in \mathbb{R}^{n \times k}$ is an orthonormal matrix ($V^T V = I_k$), since $K$ is an orthogonal projection matrix of rank $k$, its initial trace is $\text{tr}(K) = k$. Because projecting out any active column reduces the rank—and consequently the remaining trace—by exactly one regardless of which pivot is chosen, the sequence of residual traces follows a deterministic decay: $k, k-1, \dots, 1$. Therefore, for every possible permutation $\sigma \in \text{Perm}(S)$, the denominator collapses to a fixed constant:
\be
D_K(\sigma) = \prod_{i=1}^k (k - i + 1) = k!.
\ee
Substituting this constant denominator back into the sequence probability yields $\mathbb{P}_{\text{RPC}}(\sigma) = \det(K(S,S))/k!$ for every valid path. When we sum over all $k!$ permutations to evaluate the unordered subset probability \eqref{eq:rpc_subset}, the permutation count directly cancels out the denominator:
\be
\mathbb{P}_{\text{RPC}}(S) = \sum_{\sigma \in \text{Perm}(S)} \frac{\det(K(S,S))}{k!} = k! \cdot \frac{\det(K(S,S))}{k!} = \det(K(S,S)).
\ee

\subsection{Other algorithms}
Apart from ARP, volume sampling and RPCholesky, several
classical column selection algorithms can also be interpreted through DPP or the concept of volume:

\textbf{Strong rank-revealing QR}
We briefly recall the goal of strong rank-revealing QR\cite{srrqr} (strong RRQR). For given $k$, it aims to maximize $\emph{Vol}(A(:,S))$ by repeatedly interchanging two columns. In each iteration, a column in $S$ and a column outside $S$ are interchanged so that $\emph{Vol}(A(:,S))$ increases by at least $f$, where $f \geq 1$ is a parameter. Strong RRQR does not guaranty global optimality, but it ensures a locally near-maximal volume in the sense that no single column swap can increase the volume by more than 
$f$.
    \be
    \max\limits_{i\in S,j\notin S} \frac{\emph{Vol}(A(:,S\backslash\{i\}\cup \{j\}))}{\emph{Vol}(A(:,S))} \leq f.
    \ee

\begin{figure}[H]
\centering
\begin{tikzpicture}[scale=1.15]

\draw[->,thick] (-0.5,0) -- (7,0) node[right] {candidate subset $S$};
\draw[->,thick] (0,-0.2) -- (0,3.6) node[above] {$g(S)=\emph{Vol}(A(:,S))$};

\draw[line width=1.1pt,smooth,domain=0:6,samples=300]
plot(\x,{0.9*exp(-(\x-1.6)^2) + 3*exp(-(\x-4.6)^2)});

\fill[red] (1.6,0.95) circle (2.4pt);
\fill[red] (4.6,3.0) circle (2.4pt);

\draw[dashed] (1.6,0) -- (1.6,0.9);
\draw[dashed] (4.6,0) -- (4.6,2.9);

\node[below] at (1.6,0) {\small local maximizer};
\node[below] at (4.6,0) {\small global maximizer};

\node[red!80!black] at (1.7,2.85) {\small strong RRQR may select};

\draw[->,red!80!black,thin] (1.60,2.65) -- (1.60,1.3);
\draw[->,red!80!black,thin] (3.5,2.9) -- (4.4,3);

\draw[blue!70!black,thick]
(4.6,2.5) ellipse (1.2 and 0.65);

\node[blue!70!black] at (7.8,2.5)
{\small volume sampling selects w.h.p};

\end{tikzpicture}

\caption{
Optimization landscape of $g(S)=\emph{Vol}(A(:,S))$.
Strong RRQR may select subsets near either a local or global maximizer,
whereas volume sampling places higher probability mass
near the global maximizer.
}
\label{fig:rrqr_and_volume}
\end{figure}
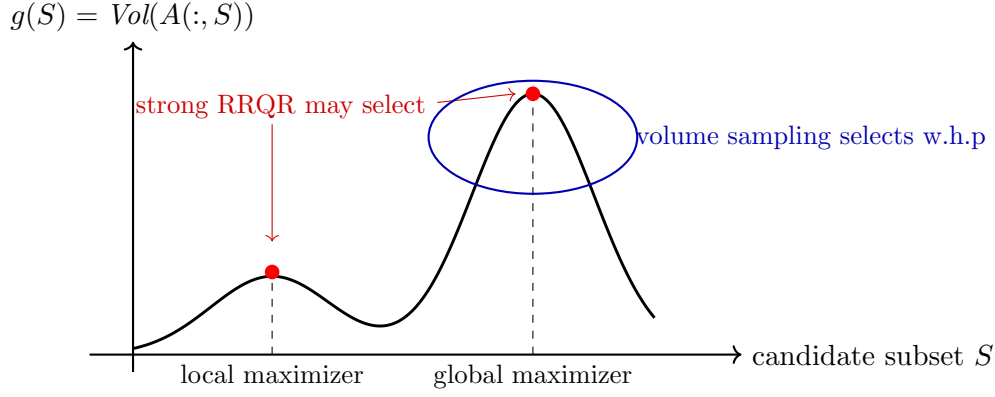

Note that fixed-size DPP sampling (or volume sampling) tends to draw a column subset with large volume with high probability, which is similar to strong RRQR. Although CSSP is inherently a discrete combinatorial optimization problem, it is helpful to view
$f(S)\coloneqq \emph{Vol}(A(:,S))$ as defining an optimization landscape. From this point of view, strong RRQR behaves like a local search method or hill-climbing method, whereas volume sampling assigns higher probability mass to regions near the global maximizer. Using the landmark of $f(S)$, we illustrate the links and difference between strong RRQR and volume sampling in Figure \ref{fig:rrqr_and_volume}.

\textbf{Column-Pivoted QR (CPQR)}
CPQR is a deterministic and iterative algorithm that selects, at each step,
the column with the largest norm and then updates the remaining columns using Gram Schmidt or a Householder reflection. Let $S_i$ be the indices selected during the first $i$ iterations.  Suppose that the selected column at iteration $i+1$ corresponds to the $j$th vector of $R_{22}$. A partial pivoted QR decomposition can be written as,
\be
A\Pi = QR = Q\begin{bmatrix}
    R_{11} & R_{12}\\
       & R_{22}
\end{bmatrix},
\ee
where $R_{11}\in \R^{i\times i}$ is upper-triangular and $R_{12}\in \R^{i\times (n-i)}$ and $R_{22}\in \R^{(m-i)\times (n - i)}$ are general matrices. We have 
\be
\frac{\emph{Vol}(A(:,S_{i+1}))}{\emph{Vol}(A(:,S_i))} = \|R_{22}(:,j)\|_2.
\ee
CPQR selects $\argmax \limits_j \|R_{22}(:,j)\|_2$, thus CPQR can be interpreted as a greedy algorithm for maximizing $\emph{Vol}(A(:,S))$. 

\textbf{Leverage score sampling:}
Leverage score sampling can be viewed as marginal approximation of a projection DPP.
If $K = VV^T$ is a rank-$k$ orthogonal projector and
$S \sim \mathrm{DPP}(K)$, then
the marginal probability \[
\mathbb{P}(i \in S) = K_{ii} = \|V(i,:)\|_2^2,
\]
coincides with the statistical leverage score.

\textbf{The algorithm of Osinsky:}
The algorithm of Osinsky in \cite{OSINSKY2025359} can be interpreted as a derandomized version of Algorithm~\ref{alg:ARPnew}
(projection DPP sampling). Since Theorem~\ref{thm:arp} says that
\[
\mathbb{E}\left [\|A - A(:,S) V(S,:)^{-T} V^T\|_F^2\right ] = (k+1) \|A - AVV^T\|_F^2,
\]
there exists a selection $S^*$ such that $\|A - A(:,S) V(S,:)^{-T} V^T\|_F^2 \leq (k+1) \|A - AVV^T\|_F^2$. The deterministic algorithm minimizes a conditional expectation at each iteration to obtain such a selection $S^*$.

The algorithms discussed above are closely related to the determinant, or equivalently the volume, of the selected columns. We conclude this section by briefly discussing two greedy methods based on different selection criteria. These methods do not have a direct DPP interpretation, but they provide useful comparisons with determinant-based column selection.
In \cite{fornace2024columnrowsubsetselection}, Fornace and Lindsey proposed the nuclear maximization method for Nystr\"{o}m approximation. In the context of the CSSP of a matrix $A$, if the corresponding partial pivoted QR factorization is written as
\be
\label{eq:pivoted_qr_partition}
A\Pi = \begin{bmatrix} Q_1 & Q_2 \end{bmatrix} \begin{bmatrix} R_{11} & R_{12} \\ 0 & R_{22} \end{bmatrix},
\ee
by Gram correspondence, their method is equivalent to greedily maximizing 
\[
\left\|Q_1^TA\right\|_F^2 = \|\begin{bmatrix}
    R_{11} & R_{12}
\end{bmatrix}\|_F^2.
\]
At each iteration, selecting a new pivot and applying Householder reflection generates a new row in the $R$-facor, the last row of $\begin{bmatrix}
    R_{11} & R_{12}
\end{bmatrix}$. Nuclear maximization chooses the pivot for which the squared norm of this new row is largest.
In \cite{Stewart1990Incremental}, Stewart proposed a method to select columns by controlling the condition number of the leading triangular submatrix in a pivoted QR factorization. If the corresponding partial pivoted QR factorization is written as in \eqref{eq:pivoted_qr_partition} and $R_{11}$ is $k\times k$, the method greedily chooses the next pivot so as to control the growth of $\left\|R_{11}^{-1}\right\|_F^2$ and determine the new pivot by the fomula:
\be
j = \argmin_{1\leq j\leq n-k} \frac{\|R_{11}^{-1}R_{12}(:,j) \|_2^2 + 1}{\|R_{22}(:,j)\|_2^2}
\ee
For recent developments on Stewart's pivoting strategy, including its strong rank-revealing properties for matrices with orthonormal rows and a randomized acceleration, we refer the reader to \cite{damle2026}.

\section{Conditional DPP for CSSP}
\label{sec:conditional_dpp}
In this section, we introduce conditional determinantal point processes (conditional DPPs), defined as random index subsets conditioned on the presence of another subset. We show how conditional DPPs can be employed within several CSSP algorithms to extend an existing column selection while preserving the columns already chosen. This setting arises naturally when certain columns must be selected in advance because of prior knowledge or application constraints, such as mandatory features in machine learning, already deployed sensors in sensor placement, or labeled samples in active learning. In particular, we combine conditional DPP and ARP and obtain a new algorithm, called multi-stage ARP. We prove that multi-stage ARP inherits the approximation guarantees of ARP. The resulting analysis reveals a connection between global DPP-based sampling strategies and some greedy algorithms.

\subsection{Conditional DPP for CSSP}
Many column subset selection algorithms such as uniform sampling, volume sampling, ARP sampling and strong rank-revealing QR are designed as one-shot procedures:
after selecting an index subset $S$, they do not naturally support augmenting the set of
chosen columns while preserving the previous selection. Conditional DPPs enable these algorithms to be extended incrementally while preserving previously selected columns. Suppose that a subset $S$ has already been selected by some strategy. Let $K\in \Rnn$ be a DPP kernel, one can sample additional indices by:
\be
U\sim \mathrm{DPP}(K)\mid S\subset U.
\ee
Such a DPP with a constraint that the random index subset $U$ must contain some elements is called a conditional DPP. A classical result in DPP theory \cite{kulesza2012dpp} states that a conditional DPP is equivalent to another DPP whose kernel is given by the Schur complement:
\be
\widetilde K = K - K(:,S)K(S,S)^{-1}K(S,:).
\ee
One can sample $T\sim \mathrm{DPP}(\widetilde K)$ and then let $U = S\cup T$.

Conditional DPP sampling can be used as a post-processing step after an initial subset $S$ has been selected by any CSSP method, such as uniform sampling, CPQR, or leverage-score sampling. It selects additional indices from the DPP conditioned on retaining all indices of $S$. In this work, we combine conditional DPP and ARP. Since ARP is equivalent to sampling from a projection DPP, the kernel needed by conditional DPP can be build from that of projection DPP. We refer to the proposed algorithm as Conditional Adaptive Randomized Pivoting(CARP), which is presented in Algorithm~\ref{alg:CARP}.

\begin{algorithm}[H]
\caption{Conditional Adaptive Randomized Pivoting (CARP)}
\label{alg:CARP}
\begin{algorithmic}[1]

\Require Initial indices $S$, matrix $V_1^{(k)} \in \mathbb{R}^{n \times k}$ (output of $\textsc{ARP-rank}(V_1)$), and matrix $V_2 \in \mathbb{R}^{n \times p}$ such that $V = [V_1,\ V_2]$ has orthonormal columns
\Ensure Indices $T = (t_1, \ldots, t_p)$
\Function{$[T,V]$=CARP}{$S,V_1^{(k)},V_2$}
\State $V \gets [V_1^{(k)} \;\; V_2] \in \mathbb{R}^{n \times (k+p)}$
\Comment{Pre-processing: orthogonalization}
\For{each $s \in S$}
    \State Apply Givens rotations to $V$ to zero out $V(s, k+1:k+p)$
\EndFor
\State $T \gets \emptyset$

\For{$i = 1, \ldots, p$}
    \State $p_j \gets \dfrac{\|V(j, k+i:k+p)\|_2^2}{p-i+1}$ for $j = 1, \ldots, n$
    \State Sample $t_i$ according to probabilities $p_j$
    \Comment{Conditional sampling}
    \State $T \gets (T, t_i)$
    
    \State $V \gets V H_i$
\EndFor

\State \Return $T,V$
\EndFunction
\end{algorithmic}
\end{algorithm}

Suppose that we have performed $\textsc{ARP-rank}(V_1)$ with 
\be
V_1\in \R^{n\times k} \text{ and } V_1^TV_1 = I,
\ee
whose output is an index subset $S$ and a matrix $V_1^{(k)}$. CARP relies on a given orthogonal matrix $V_2\in \R^{n\times p}$ to select $p$ indices. Let $V = [V_1, V_2]$, $V$ satisfy
\be
V^TV = I \text{ and } \|A - AVV^T\|_F \le \|A - AV_1V_1^T\|_F,
\ee
which means $V = [V_1,\ V_2]$ can better approximate the row space of the matrix $A$ than $V_1$. Such a matrix $V_2$ can be constructed from a sketch of the residual 
$\left(A - AV_1V_1^T\right)^T$. For instance, we generate a Gaussian 
matrix $\Omega \in \mathbb{R}^{m \times p}$ and compute
\[
B = A^T \Omega.
\]

We then orthogonalize $B$ against $V_1$ to obtain $V_2$, which is 
\[
V_2=\operatorname{orth}\left((I - V_1V_1^T)B\right).
\]
This ensures that $V_2$ lies in the orthogonal complement of $V_1$ 
while capturing information from the residual.

\def\cell{0.55}
\def\nrows{8}
\def\ncols{5}   

\newcommand{\smallgrid}{
  \foreach \i in {0,...,7}{
    \foreach \j in {0,1,2}{
      \draw[gray] (\j*\cell,-\i*\cell) rectangle ++(\cell,-\cell);
    }
  }
}
\newcommand{\biggrid}{
  \foreach \i in {0,...,7}{
    \foreach \j in {0,...,4}{
      \draw[gray] (\j*\cell,-\i*\cell) rectangle ++(\cell,-\cell);
    }
  }
}
\newcommand{\shadecols}[1]{
  \foreach \j in {0,...,#1}{
    \foreach \i in {\j,...,7}{
      \fill[pattern=north east lines]
      (\j*\cell,-\i*\cell) rectangle ++(\cell,-\cell);
    }
  }
}
\newcommand{\shadecolsall}[1]{
  \foreach \j in {#1}{
    \foreach \i in {0,...,7}{
      \fill[pattern=north east lines]
      (\j*\cell,-\i*\cell) rectangle ++(\cell,-\cell);
    }
  }
}
\newcommand{\shadecolspart}[1]{
  \foreach \j in {#1}{
    \foreach \i in {2,...,7}{
      \fill[pattern=north east lines]
      (\j*\cell,-\i*\cell) rectangle ++(\cell,-\cell);
    }
  }
}
\newcommand{\highlightrow}[3]{%
  \draw[#3, very thick] 
    ({#2*\cell},{-#1*\cell}) 
    rectangle 
    ({5*\cell},{-(#1+1)*\cell});%
}

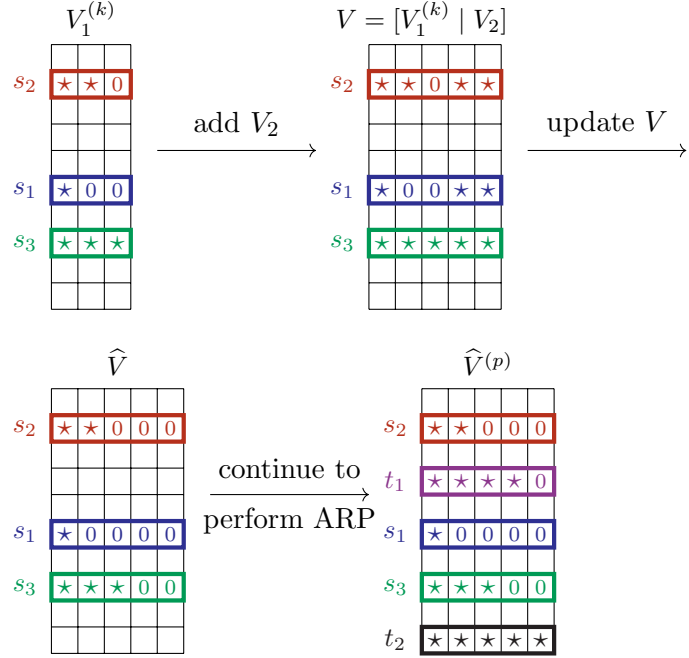
\begin{figure}[H]
    \begin{center}
        \begin{tikzpicture}[scale=.35]
            \draw (-5,-6) grid (-2,4);
            \node at (-3.5,5) {\small{$V_1^{(k)}$}};
            \node at (-6,-1.5) {\small \textcolor{Blue}{$s_1$}};
            \node at (-6,2.5) {\small \textcolor{BrickRed}{$s_2$}};
            \node at (-6,-3.5) {\small \textcolor{ForestGreen}{$s_3$}};
            \draw[ultra thick, Blue] (-5,-1) rectangle (-2,-2);
            \draw[ultra thick, BrickRed](-5,2) rectangle (-2,3);
            \draw[ultra thick, ForestGreen](-5,-3) rectangle (-2,-4);
            \node[Blue] at (-4.5,-1.5) {$\star$};
            \node[Blue] at (-3.5,-1.5) {\scriptsize $0$};
            \node[Blue] at (-2.5,-1.5) {\scriptsize $0$};
            \node[BrickRed] at (-4.5,2.5){$\star$};
            \node[BrickRed] at (-3.5,2.5){$\star$};
            \node[BrickRed] at (-2.5,2.5){\scriptsize $0$};
            \node[ForestGreen] at (-4.5,-3.5){$\star$};
            \node[ForestGreen] at (-3.5,-3.5){$\star$};
            \node[ForestGreen] at (-2.5,-3.5){$\star$};

            \node at (9,5) {\small{$V = [V_1^{(k)}\mid V_2]$}};
            \draw[->] (-1,0) -- (5,0);
            \node at (2,1) {add $V_2$};
            \draw (7,-6) grid (12,4);
            \node at (6,-1.5) {\small \textcolor{Blue}{$s_1$}};
            \node at (6,2.5) {\small \textcolor{BrickRed}{$s_2$}};
            \node at (6,-3.5) {\small \textcolor{ForestGreen}{$s_3$}};
            \draw[ultra thick, Blue] (7,-1) rectangle (12,-2);
            \draw[ultra thick, BrickRed](7,2) rectangle (12,3);
            \draw[ultra thick, ForestGreen](7,-3) rectangle (12,-4);
            \node[Blue] at (7.5,-1.5) {$\star$};
            \node[Blue] at (8.5,-1.5) {\scriptsize $0$};
            \node[Blue] at (9.5,-1.5) {\scriptsize $0$};
            \node[Blue] at (10.5,-1.5) {$\star$};
            \node[Blue] at (11.5,-1.5) {$\star$};
            \node[BrickRed] at (7.5,2.5){$\star$};
            \node[BrickRed] at (8.5,2.5){$\star$};
            \node[BrickRed] at (9.5,2.5){\scriptsize $0$};
            \node[BrickRed] at (10.5,2.5){$\star$};
            \node[BrickRed] at (11.5,2.5){$\star$};
            \node[ForestGreen] at (7.5,-3.5){$\star$};
            \node[ForestGreen] at (8.5,-3.5){$\star$};
            \node[ForestGreen] at (9.5,-3.5){$\star$};
            \node[ForestGreen] at (10.5,-3.5){$\star$};
            \node[ForestGreen] at (11.5,-3.5){$\star$};
            
            \node at (-2.5,-8) {\small{$\widehat{V}$}};
            \draw[->] (13,0) -- (19,0);
            \node at (16,1) {update $V$}; 
            \draw (-5,-19) grid (0,-9);
            \node at (-6,-14.5) {\small \textcolor{Blue}{$s_1$}};
            \node at (-6,-10.5) {\small \textcolor{BrickRed}{$s_2$}};
            \node at (-6,-16.5) {\small \textcolor{ForestGreen}{$s_3$}};
            \draw[ultra thick, Blue] (-5,-15) rectangle (0,-14);
            \draw[ultra thick, BrickRed](-5,-10) rectangle (0,-11);
            \draw[ultra thick, ForestGreen](-5,-16) rectangle (0,-17);
            \node[Blue] at (-4.5,-14.5) {$\star$};
            \node[Blue] at (-3.5,-14.5) {\scriptsize $0$};
            \node[Blue] at (-2.5,-14.5) {\scriptsize $0$};
            \node[Blue] at (-1.5,-14.5) {\scriptsize $0$};
            \node[Blue] at (-0.5,-14.5) {\scriptsize $0$};
            \node[BrickRed] at (-4.5,-10.5){$\star$};
            \node[BrickRed] at (-3.5,-10.5){$\star$};
            \node[BrickRed] at (-2.5,-10.5){\scriptsize $0$};
            \node[BrickRed] at (-1.5,-10.5){\scriptsize $0$};
            \node[BrickRed] at (-0.5,-10.5){\scriptsize $0$};
            \node[ForestGreen] at (-4.5,-16.5){$\star$};
            \node[ForestGreen] at (-3.5,-16.5){$\star$};
            \node[ForestGreen] at (-2.5,-16.5){$\star$};
            \node[ForestGreen] at (-1.5,-16.5){\scriptsize $0$};
            \node[ForestGreen] at (-0.5,-16.5){\scriptsize $0$};

            \draw[->] (1,-13) -- (7,-13);
            \node at (4,-12) {continue to};
            \node at (4,-14) {perform ARP};
            \draw (9,-19) grid (14, -9);
            \node at (11.5,-8) {\small$\widehat{V}^{(p)}$};
            \node at (8,-14.5) {\small \textcolor{Blue}{$s_1$}};
            \node at (8,-10.5) {\small \textcolor{BrickRed}{$s_2$}};
            \node at (8,-16.5) {\small \textcolor{ForestGreen}{$s_3$}};
            \node at (8,-12.5) {\small \textcolor{Fuchsia}{$t_1$}};
            \node at (8,-18.5) {\small \textcolor{Black}{$t_2$}};
            \draw[ultra thick, Blue] (9,-14) rectangle (14,-15);
            \draw[ultra thick, BrickRed](9,-10) rectangle (14,-11);
            \draw[ultra thick, ForestGreen](9,-16) rectangle (14,-17);
            \draw[ultra thick, Fuchsia](9,-12) rectangle (14,-13);
            \draw[ultra thick, Black](9,-18) rectangle (14,-19);
            \node[Blue] at (9.5,-14.5) {$\star$};
            \node[Blue] at (10.5,-14.5) {\scriptsize $0$};
            \node[Blue] at (11.5,-14.5) {\scriptsize $0$};
            \node[Blue] at (12.5,-14.5) {\scriptsize $0$};
            \node[Blue] at (13.5,-14.5) {\scriptsize $0$};
            \node[BrickRed] at (9.5,-10.5){$\star$};
            \node[BrickRed] at (10.5,-10.5){$\star$};
            \node[BrickRed] at (11.5,-10.5){\scriptsize $0$};
            \node[BrickRed] at (12.5,-10.5){\scriptsize $0$};
            \node[BrickRed] at (13.5,-10.5){\scriptsize $0$};

            \node[ForestGreen] at (9.5,-16.5){$\star$};
            \node[ForestGreen] at (10.5,-16.5){$\star$};
            \node[ForestGreen] at (11.5,-16.5){$\star$};
            \node[ForestGreen] at (12.5,-16.5){\scriptsize $0$};
            \node[ForestGreen] at (13.5,-16.5){\scriptsize $0$};

            \node[Fuchsia] at (9.5,-12.5){$\star$};
            \node[Fuchsia] at (10.5,-12.5){$\star$};
            \node[Fuchsia] at (11.5,-12.5){$\star$};
            \node[Fuchsia] at (12.5,-12.5){$\star$};
            \node[Fuchsia] at (13.5,-12.5){\scriptsize $0$};
            
            \node[Black] at (9.5,-18.5){$\star$};
            \node[Black] at (10.5,-18.5){$\star$};
            \node[Black] at (11.5,-18.5){$\star$};
            \node[Black] at (12.5,-18.5){$\star$};
            \node[Black] at (13.5,-18.5){$\star$};
        \end{tikzpicture}
    \end{center}
    \caption{Illustration of Algorithm~\ref{alg:CARP} applied to $V_1 \in \R^{10\times 3},\, V_2\in \R^{10\times 2}$ .}\label{fig:procedure_carp}
\end{figure}

CARP first performs a pre-processing stage for orthogonalization. It applies some Givens rotations to zero out the nonzero elements of the rows of $V_2$ indexed by $S$. That is,
\be
\widehat{V} = V G\text{ satisfying } \widehat{V}_{i,j} = 0 ,\ \forall i\in S,\,j>k,
\ee
where $G$ is the product of some Givens rotations (or Householder reflectors). The second stage of CARP is conditional sampling. It continues to perform norm sampling on the updated matrix $\widehat{V}$, which is similar to the operations of Algorithm \ref{alg:ARPnew}. We illustrate our algorithm in Figure~\ref{fig:procedure_carp}. In terms of complexity, the pre-processing stage costs $\mathcal{O}(nkp)$ flops and the stage of conditional sampling costs $\mathcal{O}(np^2)$ flops. Thus, the cost of Algorithm \ref{alg:CARP} becomes $\mathcal{O}(nkp + np^2)$  flops.

\begin{algorithm}[H]
\caption{Multi-Stage Adaptive Randomized Pivoting (MSARP)}
\label{alg:MSARP}
\begin{algorithmic}[1]

\Require matrix $V \in \mathbb{R}^{n \times d}$, $\mathbf{k} = [k_1,k_2,\cdots,k_t]$ satisfying $\sum\limits_{i=1}^t k_i = d$, where $k_i$ is the number of columns selected at $i$th stage and $t$ is the total number of stages
\Ensure Indices $S = (s_1,s_2,\cdots,s_d)$
\Function{$[S]$=MSARP}{$V,\mathbf{k}$}
\State $V_{old} \gets V(:,1:k_1)$
\State $n_c = k_1$
\State [$V_{old}, S$] = $\textsc{ARP-rank}(V_{old})$
\For{$i = 1:t-1$}
    \State $V_{new} \gets V(:,n_c+1:n_c+k_{i+1})$
    \State $[T,V_{old}]$ = \textsc{CARP}($S,V_{old},V_{new}$)
    \State $S \gets (S,T)$
    \State $n_c \gets n_c + k_{i+1}$
\EndFor
\State \Return $S$
\EndFunction
\end{algorithmic}
\end{algorithm}

The idea of Algorithm~\ref{alg:CARP} naturally leads to a multi-stage ARP sampling strategy. Combining Algorithm~\ref{alg:ARPnew} and Algorithm~\ref{alg:CARP} leads to an incremental selection of sets of columns algorithm, presented in Algorithm~\ref{alg:MSARP}. In this algorithm, one initially samples $k_1$ columns via ARP, followed by successive rounds of selecting $k_2, k_3, \dots, k_t$ columns through CARP.

\subsection{Analysis}
We now analyze the approximation performance of a conditional DPP. For the case that the initial selection $S$ is obtained from a general CSSP algorithm, we assume that $V(S,:)$ has full row rank. We first show that Algorithm \ref{alg:CARP} is equivalent to a conditional DPP. The Givens rotations in the preprocessing
stage construct an orthogonal matrix $G$ such that
\[
\widehat V = VG = [\,Y,\ Z\,],
\qquad
\widehat V(S,:)=[\,L,\ 0\,],
\]
where $L\in\mathbb{R}^{k\times k}$ is nonsingular and
$Z\in\mathbb{R}^{n\times p}$. Since $G$ is orthogonal,
$K=VV^T=YY^T+ZZ^T$, while
\[
K(:,S)=YL^T,\qquad K(S,S)=LL^T.
\]
Consequently,
\[
\widetilde K = K-K(:,S)K(S,S)^{-1}K(S,:)
 = YY^T+ZZ^T-YL^T(LL^T)^{-1}LY^T
 = ZZ^T.
\]
Since the last $p$ columns of $\widehat V$ form an orthonormal basis, we have shown that $\mathrm{DPP}(\widetilde K)$ is equivalent to projection $\mathrm{DPP}(ZZ^T)$, which is exactly Algorithm~\ref{alg:CARP}. We can bound the conditional expectation  of approximation errors as shown in Theorem~\ref{thm:conditional_dpp_general}.

\begin{theorem}
\label{thm:conditional_dpp_general}
    Let $A\in \Rmn$, $V \in \R^{n\times (k+p)} $ with $V^TV = I$ and $K = VV^T$. Let $S\subset \{1,2,\dots,n\}$ be an index subset with $|S| = k$ and $V(S,:)$ having  full row rank. If we sample $T\sim \mathrm{DPP}(\widetilde K)$ defined above and then let $U = S\cup T$, we have
    \be
    \label{eq:conditional_dpp_general}
    \mathbb{E}_T [ \|A - A(:,U)V(U,:)^{-T}V^T\|_F^2\mid S ] \leq (1 + p) \|\widetilde{A} -  \widetilde{A}(:,S)(V^T(:,S))^{\dagger} V^T\|_F^2,
    \ee
    where $\widetilde A = A - AVV^T$.
\end{theorem}
We construct an example in Theorem~\ref{thm:conditional-sharpness} to show that the bound is tight in the sense of conditional expectation.

\begin{theorem}
\label{thm:conditional-sharpness}
For any integers $k\geq 1$ and $p\geq 1$, there exist a matrix
$A\in\mathbb{R}^{(k+p+1)\times(k+p+1)}$, an orthonormal matrix
$V\in\mathbb{R}^{(k+p+1)\times(k+p)}$, and a prescribed initial
subset $S$ with $|S|=k$ such that the equality of~\eqref{eq:conditional_dpp_general} in
Theorem~\ref{thm:conditional_dpp_general} can be attained exactly.
\end{theorem}

\begin{proof}
Let $Q\in\mathbb{R}^{p\times(p+1)}$ satisfy
\be
QQ^\top=I_p,
\qquad
Q\mathbf{1}_{p+1}=0,
\ee
and define
\be
B=
\begin{bmatrix}
Q\\[1mm]
\dfrac{1}{2\sqrt{p+1}}\mathbf{1}_{p+1}^\top
\end{bmatrix},
\qquad
A=
\begin{bmatrix}
I_k & 0\\
0 & B
\end{bmatrix},
\qquad
V=
\begin{bmatrix}
I_k & 0\\
0 & Q^\top
\end{bmatrix}.
\ee
Let $S=\{1,\ldots,k\},
\,
K=VV^\top$,
and we sample
$T\sim\operatorname{DPP}(\widetilde K)$ as defined above.

Since $Q\mathbf{1}_{p+1}=0$ and $QQ^\top=I_p$, we have
\be
Q^\top Q
=
I_{p+1}
-
\frac{1}{p+1}
\mathbf{1}_{p+1}\mathbf{1}_{p+1}^\top.
\ee
Moreover,
\[
BB^\top
=
\begin{bmatrix}
I_p & 0\\
0 & \dfrac14
\end{bmatrix}.
\]
Hence the singular values of $A$ are $1$ with multiplicity $k+p$ and $1/2$. Therefore,
\[
\left\|A-A_{k+p}\right\|_F^2=\frac14.
\]
Moreover,
\be
\widetilde A
:=
A-AVV^\top
=
\begin{bmatrix}
0 & 0\\
0 &
\begin{bmatrix}
0\\
\dfrac{1}{2\sqrt{p+1}}\mathbf{1}_{p+1}^\top
\end{bmatrix}
\end{bmatrix},
\ee
so that
\be
\|\widetilde A\|_F^2
=
\left\|
A-A_{k+p}
\right\|_F^2= \frac14.
\ee
Since $S=\{1,\ldots,k\}$, we also have
$\widetilde A(:,S)=0$,
which gives
\be
\left\|
\widetilde A
-
\widetilde A(:,S)
\bigl(V^\top(:,S)\bigr)^\dagger
V^\top
\right\|_F^2
=
\left\|
A-A_{k+p}
\right\|_F^2.
\ee

We can simplify $K$ and $\widetilde K$ as
\be
K=
\begin{bmatrix}
I_k & 0\\
0 & Q^\top Q
\end{bmatrix},
\quad
\widetilde K
=
\begin{bmatrix}
0 & 0\\
0 & Q^\top Q
\end{bmatrix}.
\ee
Thus every choice of $T$ contains $p$ indices from the last
$p+1$ indices. Let $\widehat T\subset\{1,\ldots,p+1\}$ denote the
corresponding local indices and set
$
Q_{\widehat T}=Q(:,\widehat T).
$
Since $\operatorname{Null}(Q)=\operatorname{span}
\{\mathbf{1}_{p+1}\}$, every $p$ columns of $Q$ are linearly
independent, and hence $Q_{\widehat T}$ is nonsingular.

For $U=S\cup T$, the block structure gives
\be
A(:,U)V(U,:)^{-\top}V^\top
=
\begin{bmatrix}
I_k & 0\\
0 &
B(:,\widehat T)Q_{\widehat T}^{-1}Q
\end{bmatrix}.
\ee
Writing
\be
B(:,\widehat T)
=
\begin{bmatrix}
Q_{\widehat T}\\[1mm]
\dfrac{1}{2\sqrt{p+1}}\mathbf{1}_p^\top
\end{bmatrix},
\ee
we obtain
\be
B-
B(:,\widehat T)Q_{\widehat T}^{-1}Q
=
\begin{bmatrix}
0\\[1mm]
\dfrac{1}{2\sqrt{p+1}}
\left(
\mathbf{1}_{p+1}^\top
-
\mathbf{1}_p^\top Q_{\widehat T}^{-1}Q
\right)
\end{bmatrix}.
\ee

Let $j$ be the unique index not contained in $\widehat T$.
Since $Q\mathbf{1}_{p+1}=0$,
\be
Q(:,j)
=
-Q_{\widehat T}\mathbf{1}_p,
\ee
and therefore
\be
Q_{\widehat T}^{-1}Q(:,j)
=
-\mathbf{1}_p.
\ee
For every selected column, $Q_{\widehat T}^{-1}Q(:,i)$ is a
column of the identity matrix. Hence
\be
\mathbf{1}_{p+1}^\top
-
\mathbf{1}_p^\top Q_{\widehat T}^{-1}Q
\ee
vanishes on all indices in $\widehat T$ and equals $p+1$ at the
single omitted index $j$. Consequently,
\be
\begin{aligned}
\left\|
A-A(:,U)V(U,:)^{-\top}V^\top
\right\|_F^2
&=
\frac{1}{4(p+1)}(p+1)^2\\
&=
\frac{p+1}{4}.
\end{aligned}
\ee
This holds every choice of $T$, and therefore
\be
\mathbb{E}_T
\left[
\left\|
A-A(:,U)V(U,:)^{-\top}V^\top
\right\|_F^2
\,\middle|\,S
\right]
=
\frac{p+1}{4}.
\ee
Since $\|A-A_{k+p}\|_F^2=1/4$, we conclude that
\be
\mathbb{E}_T
\left[
\left\|
A-A(:,U)V(U,:)^{-\top}V^\top
\right\|_F^2
\,\middle|\,S
\right]
=
(p+1)\left\|A-A_{k+p}\right\|_F^2.
\ee
\end{proof}

The following results demonstrate that CARP inherits the strong multiplicative error guarantees of ARP\cite{cortinovis2025arp}.  Theorem~\ref{thm:CARP} provides the error guarantees when we do a two-stage sampling process, ARP followed with CARP. Although the bound of Theorem~\ref{thm:CARP} is slightly larger than that of Theorem~\ref{thm:arp}, our experiments show that the multi-stage ARP sampling is comparable to one-shot ARP. We put the proof of Theorem~\ref{thm:conditional_dpp_general} and Theorem~\ref{thm:CARP} in Appendix.

\begin{theorem}
\label{thm:CARP}
    Let $A \in \mathbb{R}^{m \times n}$, $V_1 \in \mathbb{R}^{n \times k}$, $V_2 \in \mathbb{R}^{n \times p}$, 
    $V = [V_1, V_2]$ satisfying $V^TV = I$, and $K = V V^T$.
    We first sample $S \sim \mathrm{DPP}(V_1 V_1^T)$, then sample $T \sim \mathrm{DPP}(\widetilde{K})$,
    where
    \[
    \widetilde{K} = K - K(:,S) K(S,S)^{-1} K(S,:).
\]
Let $U = [S, T]$. Then we have
\[
\mathbb{E}\!\left( \| A - A(:,U)  V(U,:)^{-T} V^T \|_F^2 \right)
\le (1+k)(1+p)\, \| A -  A VV^T \|_F^2.
\]
\end{theorem}

The analysis established in Theorem~\ref{thm:CARP} can be generalized to that of Algorithm~\ref{alg:MSARP}. By induction, it can be shown that the multiplicative factor in the expected error bound becomes $\prod_{i=1}^t (1+k_i)$. 

Crucially, in the extreme case where $k_1 = \dots = k_t = 1$, the bound factor simplifies to $2^t$. The power of $2$ shows up in the known worst-case exponential bounds for CPQR\cite{srrqr} as well. ARP (and projection DPPs) functions as a global sampling strategy and CPQR is a purely deterministic, greedy algorithm. By seeking informative columns while keeping previous selections, CARP incorporates a greedy element into a global strategy. Our findings thus provide a theoretical bridge between some greedy and global algorithms, quantifying the growth of the error bound when some global selection is traded for iterative, greedy selection.

\section{Fixed-Precision Adaptive Randomized Pivoting}
\label{sec:arp_tol}
We aim to address a practical limitation of the basic ARP procedure: it requires a prescribed target rank $k$, which may be difficult to determine a priori, when little is known about the spectral properties of the matrix $A$.

To overcome this issue, we combine randomized SVD with fixed precision \cite{svdsketch_fixed_precision} and ARP. In particular, we adopt an adaptive randomized range finder similar in spirit to \cite{halko_finding_2010, svdsketch_fixed_precision}, which incrementally constructs an orthonormal basis $V$ that approximates the row space of $A$. Given a block size $b$, the algorithm generates $b$ new basis vectors at each iteration and continues this process until the residual $\|A - AVV^T\|_F$ falls below a prescribed tolerance.

This leads to a fixed-precision ARP, presented in Algorithm~\ref{alg:ARP_tolerance}. Instead of fixing the rank in advance, the algorithm adaptively determines an effective rank $k$ such that the expectation of the errors is controlled in a given tolerance. We first compute $V$ by capturing the dominant row space of $A$ to the desired precision by applying an incremental randomized rangefinder to $A^T$. In Line~\ref{step:reorth}, we orthogonalize the new block $V_i$ against the existing basis $V$, ensuring that the augmented basis $[V,V_i]$ remains orthonormal. In Line~\ref{step:qb}, we compute $B_i = V_i^TA^T$ and append it as a new block row of $B$ so that $B= V^TA^T$. The loop halts when $\|(I - VV^T)A^T\|_F^2\leq \tau/(k+1)$, where $k$ is the number of columns of $V$. The standard ARP procedure is then applied to $V$ to produce an index subset $S$. For brevity, we omit some implementation details in Algorithm~\ref{alg:ARP_tolerance}. For instance, we can set a max iteration number, and the line~\ref{step:reorth} can contain several reorthogonalization steps. The guarantees of Algorithm~\ref{alg:ARP_tolerance} are formalized in Theorem~\ref{thm:arp_tol}.

\begin{algorithm}[H]
\caption{Adaptive Randomized Pivoting with Fixed Precision}
\label{alg:ARP_tolerance}
\begin{algorithmic}[1]

\Require Matrix $A \in \mathbb{R}^{m \times n}$, block size $b$, tolerance $\tau$
\Ensure Integer $k$, indices $S = (s_1, \ldots, s_k)$
\Function{$[k,S]$ = ARP-tol}{$A,b,\tau$}
\State Initialize $V \gets [\,]$, $B \gets [\,]$, $k \gets 0$, $E \gets \|A\|_F^2$

\While{$E > \tau/(k+1)$}
    \State $k \gets k + b$
    \State Draw a Gaussian matrix $\Omega \in \mathbb{R}^{m \times b}$
    
    \State $V_i \gets \mathrm{orth}\!\left(A^{T}\Omega - V(B\Omega)\right)$
    \State $V_i \gets \mathrm{orth}\!\left(V_i - V(V^T V_i)\right)$ \label{step:reorth}
    
    \State $B_i \gets V_i^T A^T$ \label{step:qb}
    \State $V \gets [V,\, V_i]$, \quad $B \gets \begin{bmatrix} B \\ B_i \end{bmatrix}$
    \State $E \gets E -  \|B_i\|_F^2$
\EndWhile
\State $[\sim,S] = \textsc{ARP-rank}(V)$ \label{step:arp}
\State \Return $k, S$
\EndFunction
\end{algorithmic}
\end{algorithm}

\begin{theorem}
\label{thm:arp_tol}
    Let $A\in\Rmn$ and $\tau > 0$. Then the random
index set $S$ returned by Algorithm \ref{alg:ARP_tolerance} satisfies
\be
\mathbb{E}\|A - A(:,S)V(S,:)^{-T}V^T\|_F^2 \leq \tau
\ee
\end{theorem}
\begin{proof}
    Let $V\in \R^{n\times k}$, using Theorem~\ref{thm:arp}, the line \ref{step:arp} of Algorithm \ref{alg:ARP_tolerance} implies that
    \be
    \label{eq:tol1}
    \mathbb{E}\|A - A(:,S)V(S,:)^{-T}V^T\|_F^2 = (k+1) \|A - AVV^T \|_F^2.
    \ee
    The loop of Algorithm \ref{alg:ARP_tolerance} guarantees that 
    \be
    \label{eq:tol2}
    \|A - AVV^T \|_F^2 \leq \frac \tau {k+1}.
    \ee
    Combining \eqref{eq:tol1} and \eqref{eq:tol2}, we get that
    \be
    \mathbb{E}\|A - A(:,S)V(S,:)^{-T}V^T\|_F^2 \leq \tau.
    \ee
\end{proof}
This bound reflects the fact that the quality of the final column subset selection is directly inherited from the accuracy of the intermediate range approximation.

\section{Numerical experiments}
\label{sec:numerics}
 
In this section, we evaluate the performance of our algorithms on several test matrices for
the column subset selection problem and the Nystr\"{o}m approximation of symmetric
positive semi-definite (SPSD) matrices. All algorithms were implemented and executed in
MATLAB R2024a. 

\subsection{Experiments of CSSP}
In this section, we show the performance of our algorithms for the column subset selection problem. For a fair comparison across different methods, we evaluate all algorithms at a common \emph{target rank} $k$, which corresponds to the number of selected columns. Let $d$ be the maximum target rank of a experiment, we use exact SVD to compute $V_d$, the dominant $d$ right singular vectors of a general matrix $A$ as input. We compare the following algorithms:
\begin{itemize}
    \item \textbf{Multi-stage ARP}: Algorithm~\ref{alg:MSARP}, applied with $V_{d}$ as mentioned above and a list $\mathbf{k} = \{k_1,k_2,...,k_t\}$ satisfying $\sum_{i=1}^t k_i = d$. The columns are selected sequentially: at stage $i$, $k_i$ new indices are sampled via CARP conditioned on the previously selected ones. The cumulative number of selected columns is
\[
k = k_1 + k_2 + \cdots + k_i.
\]
The index subset $S$ at this stage consists of all indices selected up to stage $i$.

    \item \textbf{One-shot ARP}: Algorithm~\ref{alg:ARPnew}, applied independently with the same target rank $k = k_1 + k_2 + \cdots + k_i$, using the truncated basis $V_k = V_d(:,1:k)$ to sample a subset of size $k$. In particular, for each $k$, the algorithm is rerun independently to sample a subset $S$ of size $k$. Thus, unlike multi-stage ARP, this approach does not preserve previously selected indices and instead performs a fresh sampling from the projection DPP defined by $V_k V_k^T$.

    \item \textbf{Leverage scores}: indices sampled independently with probabilities proportional to the squared row norms of $V_d$:
    \be
    \mathbb{P}(j) \propto \|V_d(j,:)\|_2^2.
    \ee
    For each $k$, the algorithm is rerun to select $k$ indices (with duplicates removed). 

    \item \textbf{CPQR}: Column-pivoted QR decomposition of $A$, retaining the first $k$ pivot columns.

    \item \textbf{Osinsky}: The deterministic algorithm of Osinsky~\cite{OSINSKY2025359} applied to $V_{k} = V_d(:,1:k)$. Same as one-shot ARP, for each $k$, the algorithm is rerun independently to sample a subset $S$ of size $k$.

    \item \textbf{SVD}: The optimal rank-$k$ approximation error $\bigl(\sum\limits_{i>k} \sigma_i(A)^2\bigr)^{1/2}$, serving as the optimal bound.
\end{itemize}

The approximation error for each method is measured as
\[
    E(S) = \|A -  Q_SQ_S^T A\|_F,
\]
where $Q_S$ is an orthonormal basis for $A(:,S)$. The approximation errors of orthogonal projectors are smaller than the errors of oblique projectors, as shown in Theorem~\ref{thm:conditional_dpp_general} and ~\ref{thm:CARP}. For the randomized methods (one-shot ARP,
multi-stage ARP, and leverage score sampling), each randomized algorithm is repeated $50$ times. We report the mean error and display error bars covering the central 80\% of outcomes,
discarding the top and bottom 10\% of trials. All errors are normalized by the Frobenius norm $\|A\|_F$.
 
\subsubsection{CSSP for synthetic matrices}
\label{subsec:cssp_synthetic}

\begin{figure}[htbp]
    \centering
    \begin{minipage}{0.48\linewidth}
        \centering
        \includegraphics[width=\linewidth]{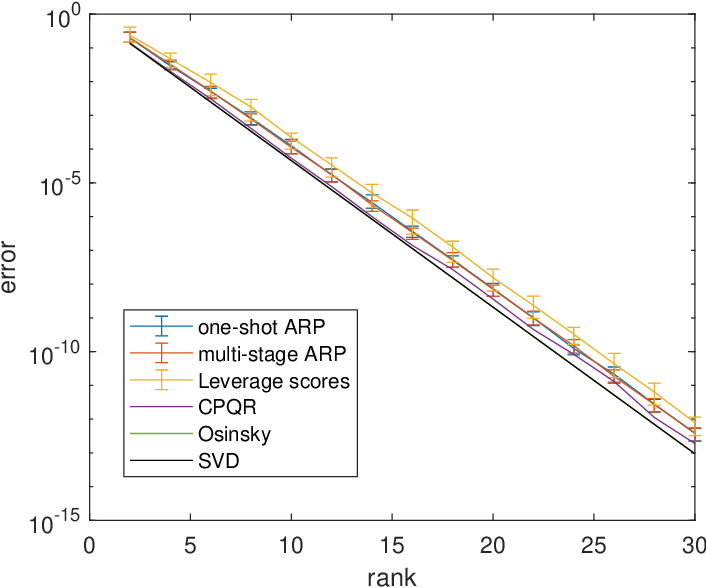}
        \caption{CSSP errors for a matrix with exponentially decaying singular
    values.}
        \label{fig:exp_decay}
    \end{minipage}
    \hfill
    \begin{minipage}{0.48\linewidth}
        \centering
        \includegraphics[width=\linewidth]{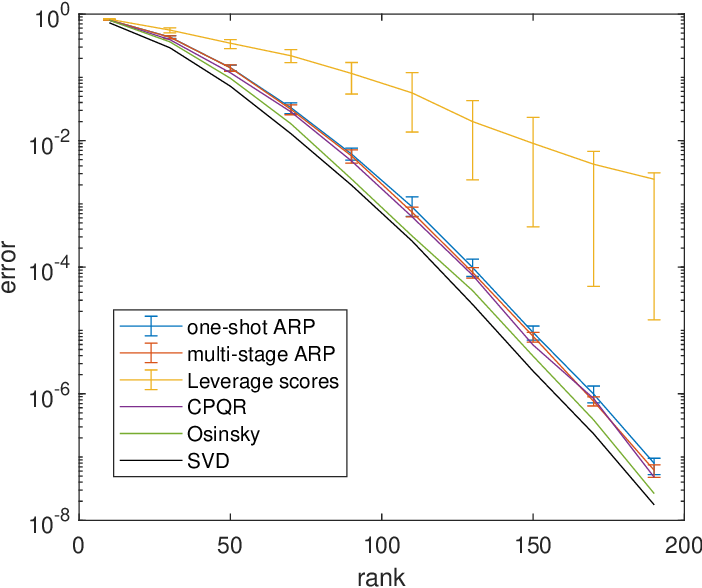}
        \caption{CSSP errors for a matrix with polynomially decaying singular
    values.}
        \label{fig:poly_decay}
    \end{minipage}
\end{figure}

We consider two families of synthetic matrices of size $2000 \times 2000$ with prescribed
spectral decay, commonly used as benchmarks in the numerical linear algebra literature.
 
\paragraph{Exponential decay.}
The first matrix $A$ has singular values decaying exponentially. We construct it by generating two random orthonormal matrix $U,V \in \R^{2000\times 2000}$ and letting $A = U\Sigma V^T$ where $\Sigma = \text{diag}\{1,e^{-1},e^{-2},\cdots,e^{-1999}\}$. Such matrices are nearly
low-rank, and most methods can exploit this structure effectively. We vary the target rank
over $k \in \{2, 4, \ldots, 30\}$, which means $k_1 = k_2 = ... = k_t = 2$ for multi-stage ARP.
 
The results are shown in Figure~\ref{fig:exp_decay}. One-shot ARP closely tracks the
optimal lower bound across all ranks, confirming the theoretical guarantee of
Theorem~\ref{thm:arp}. Multi-stage ARP performs comparably to one-shot ARP, demonstrating
that the incremental strategy of CARP does not significantly degrade quality relative to a
fresh one-shot selection. Both methods clearly outperform leverage score sampling, which
exhibits noticeably higher and more variable errors. The two deterministic ones CPQR and the algorithm of Osinsky perform well.
 
\paragraph{Polynomial decay.}
The second matrix $A$ has singular values decaying polynomially, constructed via generating two random orthonormal matrix $U,V \in \R^{2000\times 2000}$ and letting $A = U\Sigma V^T$ where $\Sigma = \text{diag}\{1,1/2^2,1/3^2,\cdots,1/1999^2\}$. Compared to the exponential case, the
singular values decay more slowly, making the problem harder at moderate ranks. We vary the
target rank over $k \in \{10, 15, \ldots, 100\}$,  which means $k_1 = 10$ and $k_2 = ... = k_t = 5$ for multi-stage ARP.
 
The results are shown in Figure~\ref{fig:poly_decay}. The qualitative picture is similar to
the exponential case, but the separation between methods is more pronounced at intermediate
ranks. One-shot ARP again closely follows the best approximation error. Multi-stage ARP
remains competitive with one-shot ARP. Both one-shot ARP and multi-stage ARP are slightly better than leverage score sampling. CPQR performs well in this matrix since the matrix structure is amenable to greedy
column selection. The algorithm of Osinsky consistently achieves near-optimal errors.

\begin{figure}[htbp]
    \centering
    \begin{minipage}{0.48\linewidth}
        \centering
        \includegraphics[width=\linewidth]{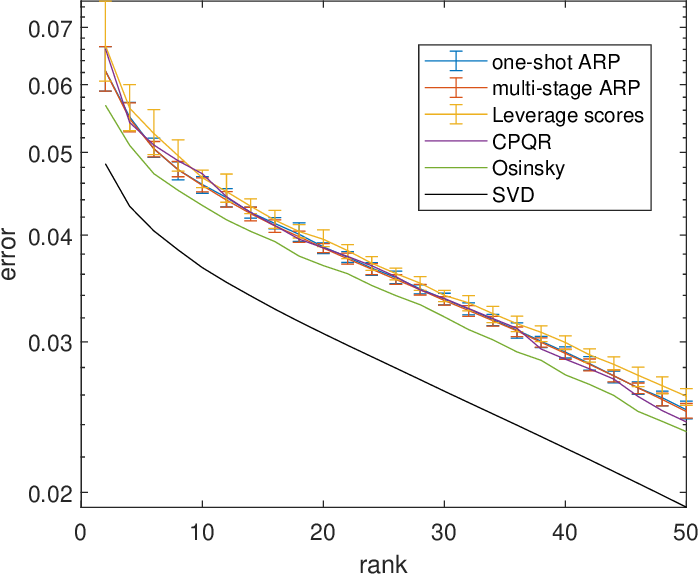}
        \caption{CSSP errors for the DNA microarray dataset.}
        \label{fig:dna}
    \end{minipage}
    \hfill
    \begin{minipage}{0.48\linewidth}
        \centering
        \includegraphics[width=\linewidth]{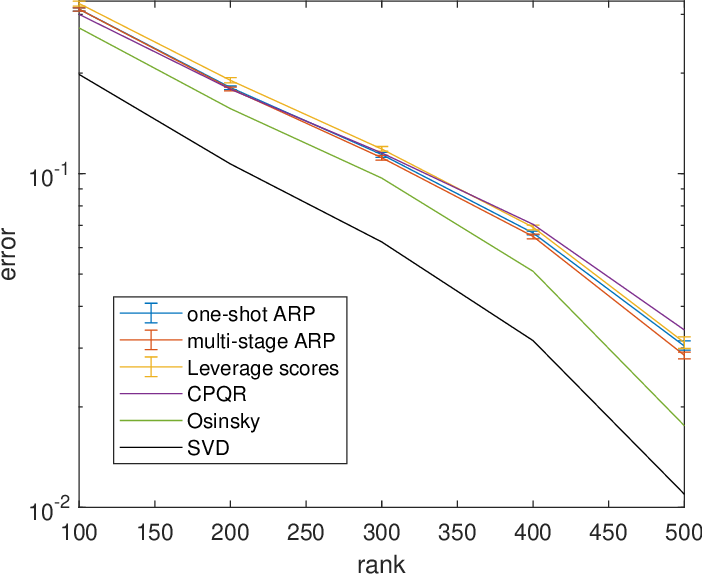}
        \caption{CSSP errors for the mnist dataset matrix.}
        \label{fig:mnist}
    \end{minipage}
\end{figure}

\subsubsection{CSSP for real-world data}
\label{subsec:cssp_real}
\paragraph{DNA microarray data.}
The first dataset is a DNA microarray matrix\footnote{Downloaded from \url{https://ftp.ncbi.nlm.nih.gov/geo/series/GSE10nnn/GSE10072/matrix/}.}, which contains gene expression measurements and serves as a representative example of
data arising in computational biology, where column subset selection can be used to identify
a small number of informative genes. The matrix has dimensions $m \times n$ and we vary the
target rank over $k \in \{2, 4, \ldots, 50\}$, which means that we set $t = 25$ and $k_1 = k_2 = ... = k_t = 2$ for multi-stage ARP.
 
The results are displayed in Figure~\ref{fig:dna}. The behavior is broadly similar to the
synthetic experiments. One-shot ARP achieves errors close to the optimal lower bound, while
multi-stage ARP is competitive and only slightly worse on average. The variance of leverage
score sampling is noticeably larger than that of ARP, and its mean error exceeds that of ARP
across most ranks. CPQR and Osinsky are again among the top-performing deterministic methods.
The real-world matrix has an irregular spectral profile, and the DPP-based methods handle
this gracefully by adapting the sampling probabilities at each step according to the current
residual subspace.

\paragraph{MNIST handwritten digits.}
The second dataset is the MNIST handwritten digit collection. We use $500$ images per digit
class (digits 0--9), yielding $n = 5000$ images in total. Each image is of size $28 \times 28$
pixels, which we flatten into a vector of length $784$, forming a data matrix
$A \in \mathbb{R}^{784 \times 5000}$. Column subset selection on $A$ corresponds to selecting
a small set of representative images whose linear span approximates the full dataset. We vary
the target rank over $k \in \{20, 40, \ldots, 300\}$, which means that we set $t = 15$ and $k_1 = k_2 = ... = k_t = 20$ for multi-stage ARP.
 
The results are shown in Figure~\ref{fig:mnist}. The MNIST matrix has a slowly decaying
singular value spectrum — a characteristic of real image data — which makes column subset
selection at large ranks non-trivial. One-shot ARP and multi-stage ARP perform nearly
identically across all ranks, where multi-stage ARP is slightly better. Leverage scores show slightly higher errors on average but
remain competitive throughout the rank range. CPQR behaves well at low ranks ($k \lesssim 60$),
and produces competitive approximation errors.
However, as the target rank increases, CPQR progressively deteriorates relative to all other
methods, eventually becoming the worst-performing algorithm in the comparison. 
In contrast, the ARP-based methods maintain a stable and near-optimal performance throughout,
illustrating the practical advantage of DPP-based global selection over greedy strategies at
larger ranks. The deterministic algorithm of Osinsky achieves the best errors among all methods
across the entire rank range.

\subsection{Experiments of Nystr\"{o}m approximation for SPSD matrices}
\label{subsec:nystrom}
In this section, we show the performance of our algorithms for Nystr\"{o}m approximations. As is shown in \cite{epperly2024gram}, using Gram correspondence, each strategy of CSSP for a matrix $A$ can be naturally applied to Nystr\"{o}m approximation for a  SPSD matrix $M = A^TA$. We still evaluate all algorithms at a common target rank $k$. For ARP-based algorithms(Algorithm~\ref{alg:ARPnew} and Algorithm~\ref{alg:MSARP} and the Algorithm of Osinsky), we compute the dominant $d$ eigenvectors $V_d$ of $M$, where $d$ is the maximum target rank of a experiment. We compare the following algorithms:

\begin{figure}[t]
    \centering
    \begin{subfigure}[b]{0.48\linewidth}
        \includegraphics[width=\linewidth]{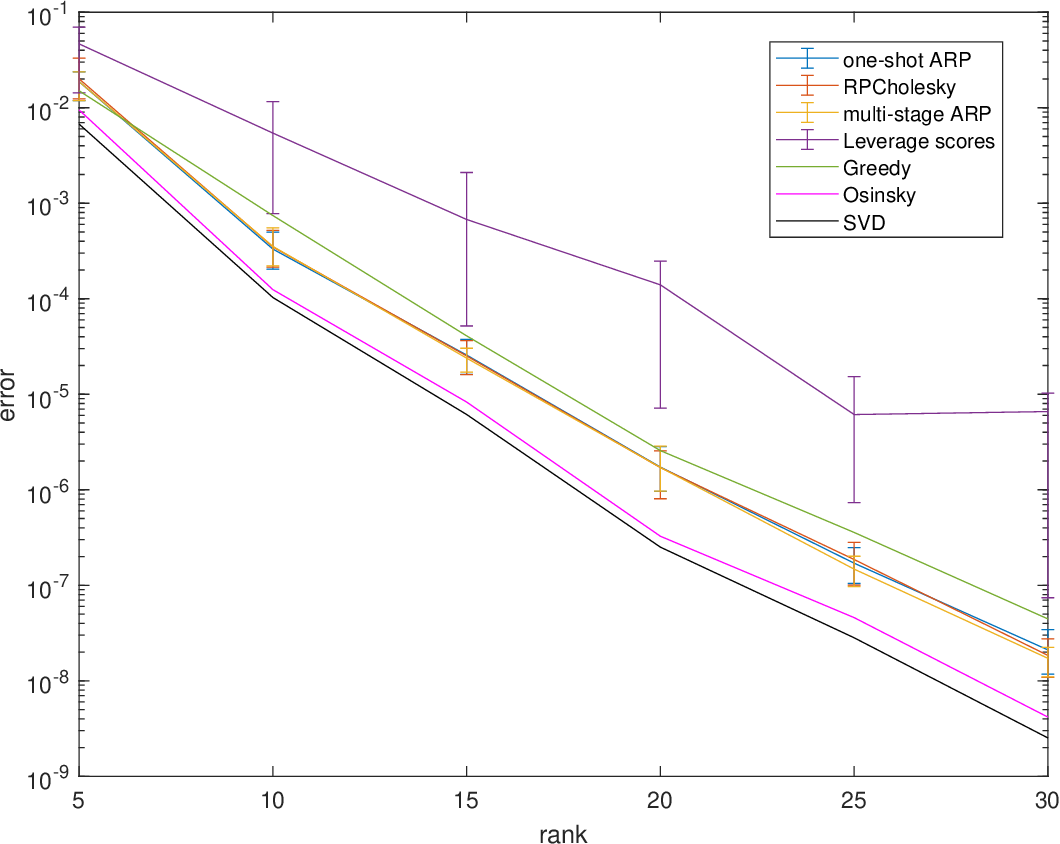}
        \caption{Nystr\"{o}m errors, two-moons kernel.}
        \label{fig:nystrom_moons}
    \end{subfigure}
    \hfill
    \begin{subfigure}[b]{0.48\linewidth}
        \includegraphics[width=\linewidth]{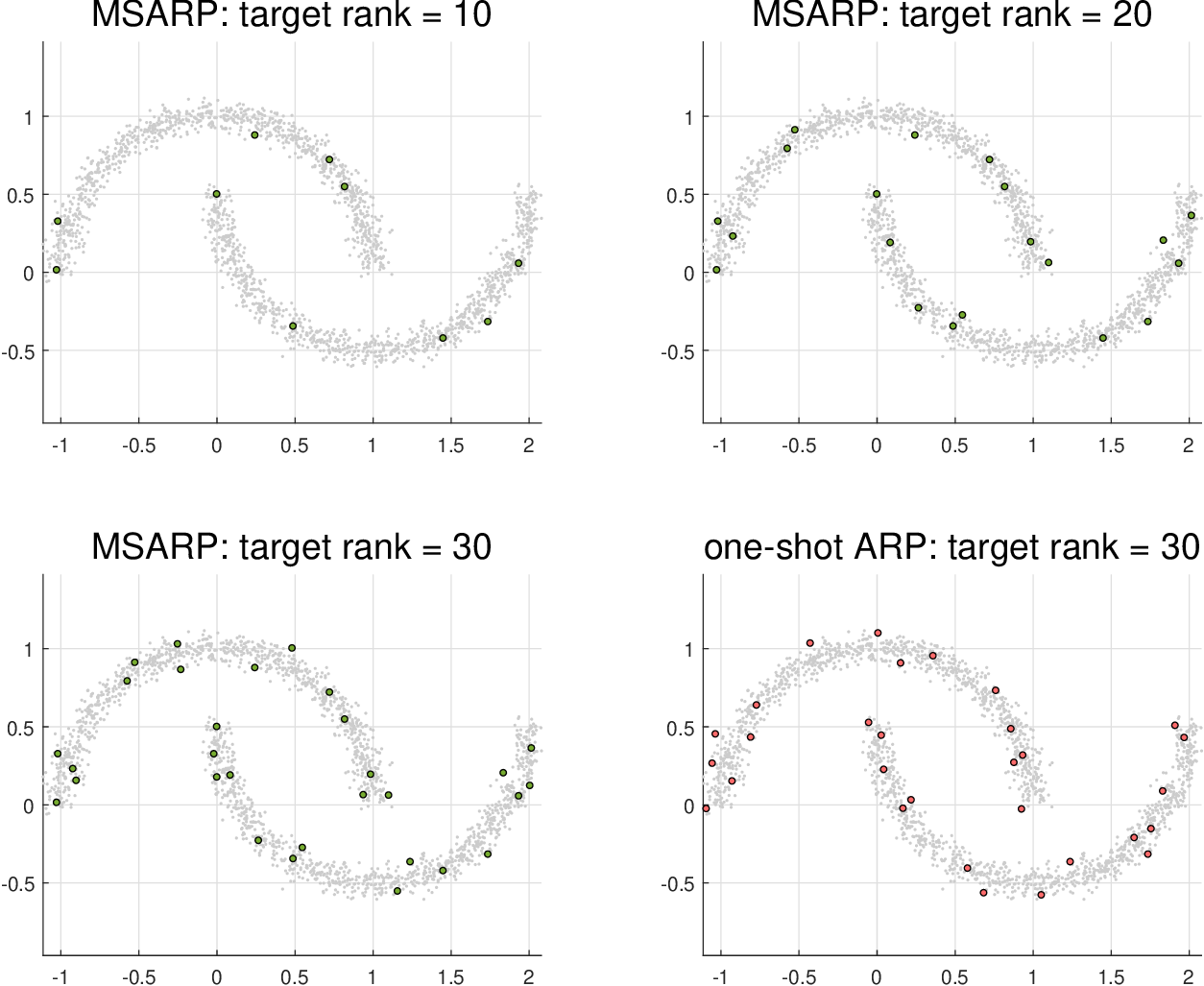}
        \caption{Point locations, two-moons kernel.}
        \label{fig:nystrom_moons_pivots}
    \end{subfigure}
    \caption{Nystr\"{o}m approximation of the two-moons kernel ($n = 2000$, bandwidth $= 2$,
    noise $= 0.05$). Left: trace-norm error vs.\ target rank. Right: selected points from
    MSARP at ranks 10, 20, 30 and from one-shot ARP at rank 30.}
    \label{fig:moons_combined}
\end{figure}

\begin{figure}[t]
    \centering
    \begin{subfigure}[b]{0.48\linewidth}

        \includegraphics[width=\linewidth]{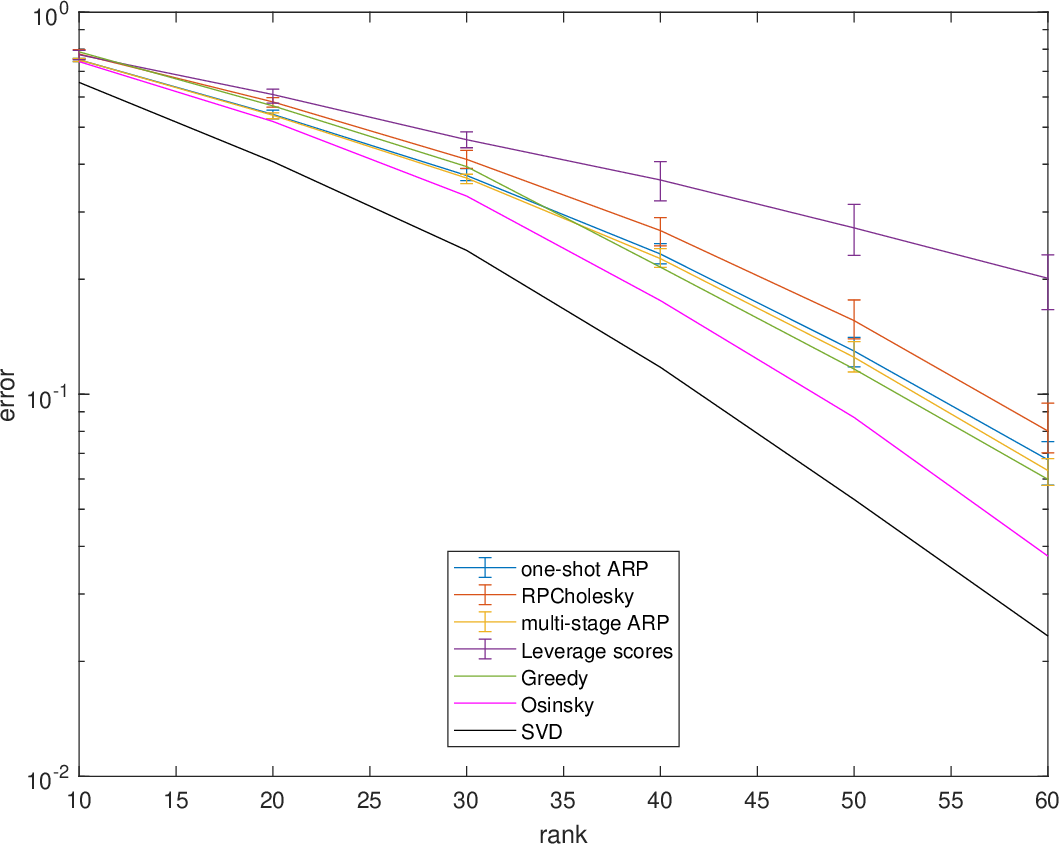}
        \caption{Nystr\"{o}m errors, two-smiles kernel.}
        \label{fig:nystrom_twosmiles}
    \end{subfigure}
    \hfill
    \begin{subfigure}[b]{0.48\linewidth}
        \includegraphics[width=\linewidth]{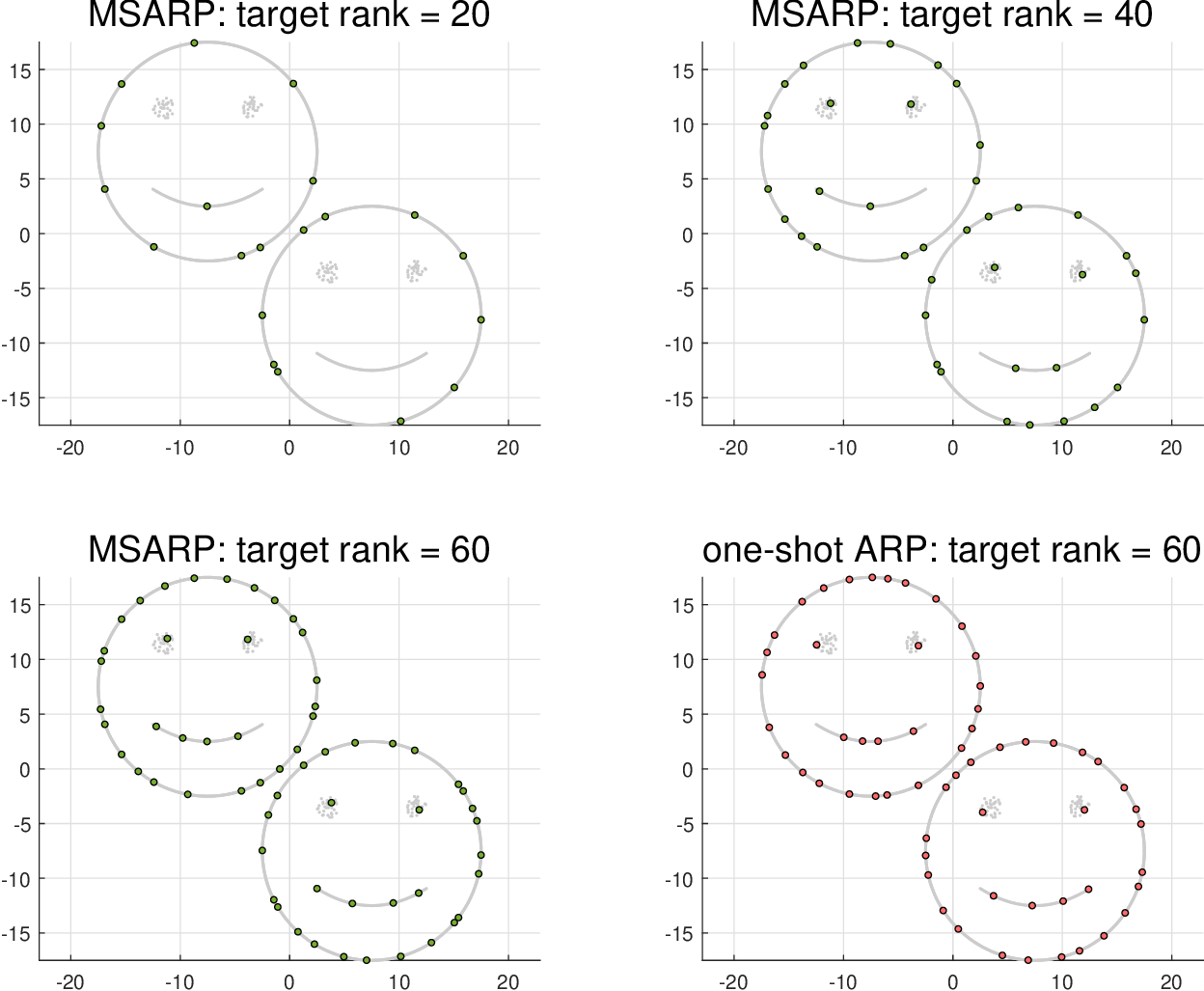}
        \caption{Point locations, two-smiles kernel.}
        \label{fig:nystrom_twosmiles_pivots}
    \end{subfigure}
    \caption{Nystr\"{o}m approximation of the two-smiles kernel ($n = 2000$, bandwidth $= 2$).
    Left: trace-norm error vs.\ target rank. Right: selected points from MSARP at ranks
    20, 40, 60 and from one-shot ARP at rank 60.}
    \label{fig:twosmiles_combined}
\end{figure}

\begin{itemize}
    \item \textbf{Multi-stage ARP}: Algorithm~\ref{alg:MSARP}, applied with $V_{d}$ as mentioned above and a list $\mathbf{k} = \{k_1,k_2,...,k_t\}$ satisfying $\sum_{i=1}^t k_i = d$. The columns are selected sequentially: at stage $i$, $k_i$ new indices are sampled via CARP conditioned on the previously selected ones. The cumulative number of selected columns is
\[
k = k_1 + k_2 + \cdots + k_i.
\]
The subset $S$ at this stage consists of all indices selected up to stage $i$.

    \item \textbf{One-shot ARP}: Algorithm~\ref{alg:ARPnew}, applied independently with the same target rank $k = k_1 + k_2 + \cdots + k_i$, using the truncated basis $V_k = V_d(:,1:k)$ to sample a subset of size $k$. In particular, for each $k$, the algorithm is rerun independently to sample a subset $S$ of size $k$. Thus, unlike multi-stage ARP, this approach does not preserve previously selected indices and instead performs a fresh sampling from the projection DPP defined by $V_k V_k^T$.

    \item \textbf{Leverage scores}: indices sampled independently with probabilities proportional to the squared row norms of $V_d$:
    \be
    \mathbb{P}(j) \propto \|V_d(j,:)\|_2^2.
    \ee
    For each $k$, the algorithm is rerun to select $k$ indices (with duplicates removed). 

    \item \textbf{RPCholesky}: Algorithm~\ref{alg:rpcholesky}, proposed in \cite{RPCholesky}, applied with $M$ and $d$, retaining the first $k$ pivots.
    \item \textbf{Greedy}: the Nystr\"{o}m approximation version of CPQR, each iteration selects the index of largest diagonal entry of $M$, followed with a rank-one update, retaining the first $k$ pivots.

    \item \textbf{Osinsky}: The deterministic algorithm of Osinsky~\cite{OSINSKY2025359} applied to $V_{k} = V_d(:,1:k)$. Same as one-shot ARP, for each $k$, the algorithm is rerun independently to sample a subset $S$ of size $k$.

    \item \textbf{SVD}: The  approximation nuclear norm error of truncated SVD: $\sum\limits_{i>k} \lambda_i(M)$, serving as the optimal bound.
\end{itemize}

Given a set of landmark indices $S \subset \{1,\ldots,n\}$
with $|S| = k$, 
the quality of Nyst\"{o}rm approximation is measured by the trace-norm error
\[
    E(S) = \|M - M(:,S)M(S,S)^{-1}M(S,:)\|_* = \mathrm{tr}\!\left(M - M(:,S)M(S,S)^{-1}M(S,:)\right).
\]
For the randomized methods (one-shot ARP,
multi-stage ARP, and leverage score sampling), each randomized algorithm is repeated $50$ times. We report the mean error and display error bars covering the central 80\% of outcomes,
discarding the top and bottom 10\% of trials. All errors are normalized by $\mathrm{tr}(M)$. 

\paragraph{Two-moons kernel.}
We construct a Gaussian kernel matrix $M \in \mathbb{R}^{n \times n}$ with $n = 2000$ data
points arranged in two interlacing semicircles\footnote{adapted from \url{https://scikit-learn.org/stable/modules/generated/sklearn.datasets.make_moons.html}.} The first coordinate of each data point is added with Gaussian noise $\epsilon \sim \mathcal{N}(0,0.05)$. The bandwidth is set as $2$. We let target ranks
$k$ vary over $\{5, 10, \ldots, 30\}$, which means we set $t = 6$ and $k_1 = k_2 = \cdots = k_t = 5$ for multi-stage ARP. The results in Figure~\ref{fig:nystrom_moons} show that one-shot ARP,
multi-stage ARP and RPCholesky achieve nearly identical errors and work slightly better than the greedy method.
In contrast, leverage scores exhibit larger errors and variance. The algorithm of Osinsky performs nearly optimally at each target rank.
 
Figure~\ref{fig:nystrom_moons_pivots} shows that multi-stage ARP incrementally selects points in both semicircles as the target rank
increases, which have similar diversity shown in the selection from one-shot ARP at rank 30.

\paragraph{Two-smiles kernel.}
The second kernel is constructed from $n = 2000$ points forming two smile shapes
in $\mathbb{R}^2$ (one above the other), with the same Gaussian bandwidth $\sigma = 2$ and
target ranks $k \in \{10, 20, \ldots, 60\}$, which means we set $t = 6$ and $k_1 = k_2 = \cdots = k_t = 10$ for multi-stage ARP.
 
The results of approximation errors are shown in
Figure~\ref{fig:nystrom_twosmiles}. One-shot ARP, multi-stage ARP, and the greedy method
behave similarly and are slightly better than 
RPCholesky. The algorithm Osinsky achieves slightly better errors at larger ranks, while leverage scores
show larger errors when the target rank increases. The plots in Figure~\ref{fig:nystrom_twosmiles_pivots} shows the process of how MSARP selects points and compares it with the selection from one-shot ARP at target rank $60$, which confirms that MSARP successfully
discovers both smile structures. At rank $20$, selected points are already spread across both smile
arcs. By rank $60$, the coverage of both arcs and the four eye clusters is visually well
balanced. MSARP can allocate new landmarks to under-represented regions rather than
resampling all at once.

\section*{Acknowledgments}
This project has received funding from the European Research
Council (ERC) under the European Union's Horizon 2020 research and innovation program (grant agreement No 810367). During the preparation of this manuscript, the authors used ChatGPT 5.2 (OpenAI) and Gemini 3.1 (Google) solely for language editing and improving the clarity and readability of the text, and a final proofreading.
\appendix

\section{Proof of Theorem~\ref{thm:conditional_dpp_general} and Theorem~\ref{thm:CARP}}
To avoid writing many transposes, we rewrite Theorem \ref{thm:CARP} in the form of selecting rows from the matrix $A\in \mathbb{R}^{m \times n}$. We firstly introduce a lemma\cite{OSINSKY2025359} discovered by Osinsky, which is useful for later derivation.

\begin{lemma}
\label{lma:osinsky}
Let $Q\in \R^{m\times k}$ be an orthogonal matrix and $S = (s_1,s_2,\cdots,s_k)$ be an index subset of $[\![m]\!]$. Defining an oblique projector $\Pi_S = I - Q(E_S^TQ)^{-1}E_S^T$, we have
\be \label{eq:osinsky_formula}
\Pi_S (I - QQ^T) = \Pi_S
\ee
\end{lemma}
\begin{proof}
We use straightforward calculation to verify \eqref{eq:osinsky_formula}:
\[
\begin{aligned}
        \Pi_S (I - QQ^T) &=I - QQ^T - Q(E_S^TQ)^{-1}E_S^T + Q(E_S^TQ)^{-1}(E_S^TQ)Q^T\\ 
        &= I - Q(E_S^TQ)^{-1}E_S^T\\
        &= \Pi_S
    \end{aligned}
\]
\end{proof}

Before proving Theorem~\ref{thm:CARP}, we need to introduce another lemma (see \cite{lsq_volumesampling} for proof), which quantifies the expectation of errors of least-square problems solved by volume sampling.

\begin{lemma}
\label{lma:dpp_lsq}
    Let $X\in \R^{m\times k}$ has full-rank and $y\in \R^m$. If we sample $k$ indices $S\sim VS_k(X^T)$, we can define an approximate solution
    $$\widehat \beta \coloneqq X(S, :)^{-1}y(S)$$ of the least square problem $\min\limits_{\beta \in \R^{k}} \|X\beta - y \|_2$. We have
    \be
        &\mathbb{E} \widehat \beta = X^{\dagger}y\\
        &\mathbb{E} (\|X \widehat \beta - y \|_2^2) \leq (k+1) \|(I - XX^{\dagger})y\|_2^2.
    \ee
\end{lemma}

\begin{theorem}
\label{thm:main_result_for_conditionalDPP}
Let $A \in \mathbb{R}^{m \times n}$, $Q \in \mathbb{R}^{m \times (k+p)}$ satisfying $Q^TQ = I$, $K = QQ^T$ and $S$ be an $k$-index subset of $\{1,2,\cdots,m\}$ with $Q(S,:)$ having full row rank. We sample $T\sim \mathrm{DPP}(\widetilde K)$,
    where
    \[
    \widetilde K = K - K(:,S) K(S,S)^{-1} K(S,:).
\]
Let $U = [S, T]$, then we have
\be
\mathbb{E}\!\left( \| A - QQ(U,:)^{-1} A(U,:) \|_F^2  \mid S \right)
\leq (1+p) \| \widetilde{A} - QQ(S,:)^{\dagger}\widetilde{A}(S,:) \|_F^2,
\ee
where $\widetilde A = (I - QQ^T)A$.
\end{theorem}

\begin{proof}
    Let
\be
\widetilde A = (I - Q Q^T) A.
\ee
By Lemma \ref{lma:osinsky},
\be
(I - Q Q(U,:)^{-1} E_U^T) A
= (I - Q Q(U,:)^{-1} E_U^T)\widetilde A.
\ee

Let
\be
\label{eq:def_wu}
W_U = Q(U,:)^{-1} \widetilde A(U,:),
\ee
then
\be
\label{eq:pf1}
\begin{aligned}
  \| \widetilde A - Q Q(U,:)^{-1} \widetilde A(U,:) \|_F^2 &= \|\widetilde A - Q W_U\|_F^2\\
&= \|\widetilde A\|_F^2 + \|Q W_U\|_F^2\\
&= \|\widetilde A\|_F^2 + \|W_U\|_F^2.  
\end{aligned}
\ee
The second equality is derived from that $range(\widetilde A ) \subset range(I - QQ^T )$, which is orthogonal to $Q$.
Let
\be
\label{eq:def_w1}
W_1 = Q(S,:)^\dagger \widetilde A(S,:)
= Q(S,:)^T (Q(S,:)Q(S,:)^T)^{-1} \widetilde A(S,:),
\ee
and define
\be
\label{eq:def_r}
R = \widetilde A - Q W_1.
\ee

Since $Q(S,:) \in \mathbb{R}^{k \times (k+p)}$ has a null space of dimension $p$,
let $B \in \mathbb{R}^{(k+p)\times p}$ be an orthonormal basis of $\mathrm{Null}(Q(S,:))$,
i.e.,
\be \label{eq:orthogonality}
B^T B = I_p, \quad Q(S,:) B = 0.
\ee

Let
\be
\label{eq:def_w2}
W_2 = B (Q(T,:)B)^{-1} R(T,:), 
\ee
we first show that $W_U$ can be split into $W_1$ and $W_2$.

By \eqref{eq:def_w1} and \eqref{eq:orthogonality}, we have
\be
\label{eq:qs}
Q(S,:)W_1 = \widetilde A(S,:), \quad Q(S,:)W_2 = 0,
\ee
and using \eqref{eq:def_w2}, we have
\be
Q(T,:)W_2 = R(T,:) = \widetilde A(T,:) - Q(T,:)W_1,
\ee
which implies
\be
\label{eq:qt}
Q(T,:)(W_1 + W_2) = \widetilde A(T,:).
\ee

Thus by \eqref{eq:qs} and \eqref{eq:qt}, we have 
\be
\label{eq:qu}
Q(U,:)(W_1 + W_2) =  \begin{bmatrix}
&Q(S,:)(W_1 + W_2)\quad\\&Q(T,:)(W_1 + W_2)\quad
\end{bmatrix} = \widetilde A(U,:).
\ee

Combining \eqref{eq:def_wu} and \eqref{eq:qu}, we obtain
\be
\label{eq:split_wu}
W_1 + W_2 = W_U
\ee
Note that the columns of $W_1$ lie in the span of $Q(S,:)^T$, while the columns of $W_2$ lie in the span of $B$, hence $range(W_1)$ is orthogonal to $range(W_2)$

Therefore,
\be
\label{eq:norm_wu}
\|W_U\|_F^2 = \|W_1\|_F^2 + \|W_2\|_F^2. 
\ee

We compute $\mathbb{E}_T[\|W_2\|_F^2 \mid S]$.

Consider a least square problem
\be
\min\limits_{M\in \R^{p\times n}} \|Q B M - R\|_F^2,
\ee
Sampling rows indexed by $T\sim VS_p(QB)$ is equivalent to sampling ones by $T\sim DPP(\widetilde K)$ and we obtain the approximation:
\be
\label{eq:hatm}
\hat M = (Q(T,:)B)^{-1} R(T,:), \quad W_2 = B \hat M.
\ee

By Lemma \ref{lma:dpp_lsq}, 
\be
\label{eq:condexp_w2}
\mathbb{E}_T\!\left( \|Q B \widehat M - R\|_F^2 \mid S \right) \leq \mathbb{E}_T\!\left( \|Q W_2 - R\|_F^2 \mid S \right)
= (p+1)\|R\|_F^2.
\ee

Using \eqref{eq:def_w1},\eqref{eq:def_r} and \eqref{eq:orthogonality}. we get
\be
\label{eq:orthogonality2}
R^T Q B = \widetilde A^T Q B - W_1^T Q^T Q B = -(B^T W_1)^T = 0.
\ee

Thus by \eqref{eq:def_w2}, \eqref{eq:hatm} and \eqref{eq:orthogonality2},
\be \label{eq:norm_w2}
\begin{aligned}
\mathbb{E}_T(\|W_2\|_F^2 \mid S) &= \mathbb{E}_T(\|QW_2\|_F^2 \mid S)\\
&=\mathbb{E}_T(\|QB\widehat M\|_F^2 \mid S)\\
&=\mathbb{E}_T(\|QB\widehat M - R\|_F^2 \mid S) - \|R\|_F^2\\
&\leq p\|R\|_F^2\\
&= p\|\widetilde A -QW_1\|_F^2.
\end{aligned}
\ee
Combining \eqref{eq:pf1},\eqref{eq:def_w1},\eqref{eq:split_wu}and \eqref{eq:norm_w2}, we can end the proof:
\be
\begin{aligned}
    \mathbb{E}_T(\| \widetilde A - Q Q(U,:)^{-1} \widetilde A(U,:) \|_F^2\mid S) &= \|\widetilde A\|_F^2 + \mathbb{E}_T(\| W_U\|_F^2\mid S)\\
   &=\|\widetilde A\|_F^2 + \| W_1\|_F^2 + \mathbb{E}_T(\| W_2\|_F^2\mid S) \\
   & \leq \|\widetilde A\|_F^2 + \| W_1\|_F^2 + p\|\widetilde A -QW_1\|_F^2\\
   & = (p+1)\|\widetilde A -QW_1\|_F^2\\
   & = (p+1)\|\widetilde A -QQ(S,:)^\dagger \widetilde A(S,:)\|_F^2.
\end{aligned}
\ee
\end{proof}

\begin{theorem}
\label{thm:main_result}
    Let $A \in \mathbb{R}^{m \times n}$, $Q_1 \in \mathbb{R}^{m \times k}$, $Q_2 \in \mathbb{R}^{m \times p}$, $Q = [Q_1, Q_2]$ satisfy $Q^TQ = I$ and $K = QQ^T$. We firstly sample $S\sim DPP(Q_1Q_1^T)$ and then sample $T \sim \mathrm{DPP}(\widetilde K)$,
    where
    \[
    \widetilde K = K - K(:,S) K(S,S)^{-1} K(S,:).
\]
Let $U = [S, T]$. Then we have
\be
\mathbb{E}\!\left( \| A - QQ(U,:)^{-1} A(U,:) \|_F^2  \right)
\le (1+p)(1+k) \| A - QQ^TA \|_F^2.
\ee
\end{theorem}
\begin{proof}   
Following the proof of Theorem~\ref{thm:main_result_for_conditionalDPP}, we have
\be
 \mathbb{E}_T(\| \widetilde A - Q Q(U,:)^{-1} \widetilde A(U,:) \|_F^2\mid S) =\|\widetilde A\|_F^2 + \| W_1\|_F^2 + \mathbb{E}_T(\| W_2\|_F^2\mid S)
\ee
We now estimate $\mathbb{E}_S\|W_1\|_F^2$.

\be
\|W_1\|_F^2
= \|Q(S,:)^\dagger \widetilde A(S,:)\|_F^2
= \mathrm{tr}\!\left( \widetilde A(S,:)^T (Q(S,:)Q(S,:)^T)^{-1} \widetilde A(S,:) \right).
\ee

Since $Q = [Q_1, Q_2]$, then
\be
Q(S,:) Q(S,:)^T = Q_1(S,:)Q_1(S,:)^T + Q_2(S,:)Q_2(S,:)^T
\succeq Q_1(S,:)Q_1(S,:)^T,
\ee
hence
\be
(Q(S,:) Q(S,:)^T)^{-1} \preceq (Q_1(S,:)Q_1(S,:)^T)^{-1}.
\ee

Thus
\be
\|W_1\|_F^2
\le \|Q_1(S,:)^{-1} \widetilde A(S,:)\|_F^2.
\ee

Since $S \sim \mathrm{DPP}(Q_1Q_1^T) = VS_k(Q_1^T)$, using Lemma \ref{lma:dpp_lsq} again, we have
\be
\mathbb{E}_S \|\widetilde A - Q_1 Q_1(S,:)^{-1} \widetilde A(S,:)\|_F^2
\leq (k+1)\|\widetilde A\|_F^2.
\ee
By 
\be
\label{eq:norm_w1}
\begin{aligned}
\mathbb{E}_S \|W_1\|_F^2 &\le \mathbb{E}_S \|Q_1(S,:)^{-1} \widetilde A(S,:)\|_F^2 \\
&=\mathbb{E}_S \|Q_1Q_1(S,:)^{-1} \widetilde A(S,:)\|_F^2\\
&=\mathbb{E}_S \|\widetilde A - Q_1Q_1(S,:)^{-1} \widetilde A(S,:)\|_F^2 - 
\|\widetilde A\|_F^2 \\
&\leq k \|\widetilde A\|_F^2    
\end{aligned}
\ee

\medskip

Combining \eqref{eq:pf1}, \eqref{eq:norm_wu}, \eqref{eq:norm_w2} and \eqref{eq:norm_w1}, and taking total expectation, we get
\[
\mathbb{E}\|\widetilde A - Q W_U\|_F^2
\le \|\widetilde A\|_F^2 (1 + k + p(1+k))
= (1+k)(1+p)\|\widetilde A\|_F^2.
\]
\end{proof}

 \bibliographystyle{plain}
\bibliography{reference}
\end{document}